\documentclass[a4paper,10pt,hidelinks]{amsart}
\usepackage{graphicx} 
\usepackage[a4paper, margin=2 cm]{geometry}
\usepackage{hyperref}
\usepackage{amsthm}
\usepackage{todonotes}
\usepackage{thm-restate}
\usepackage{geometry}
\usepackage{amssymb}
\usepackage{amsmath}
\usepackage{comment}
\usepackage{xcolor}
\usepackage{blkarray}
\usepackage{lmodern}
\usepackage[english]{babel}
\usepackage{float}
\usepackage{tikz}
\usepackage{enumitem}
\usetikzlibrary{arrows.meta,arrows}
\usetikzlibrary{arrows,decorations.markings}
\usetikzlibrary{decorations.pathmorphing}

\usetikzlibrary{shapes.geometric}
\usetikzlibrary{calc}

\usepackage[square,sort,comma,numbers]{natbib}
\usepackage{varioref}

\usepackage{datetime}

\usepackage{mathtools}

\usepackage[h]{esvect}

\usepackage{array}

\usepackage{listings}
\usepackage{microtype}

\usepackage{booktabs}

\usepackage{subcaption}

\newcommand{\Z}{\mathbb{Z}}

\renewcommand{\epsilon}{\ensuremath\varepsilon}

\renewcommand{\phi}{\ensuremath{\varphi}}

\newtheorem{theorem}{Theorem}[section]
\newtheorem{lemma}[theorem]{Lemma}

\newtheorem{claim}[theorem]{Claim}
\newtheorem*{claim*}{Claim}
\newtheorem{proposition}[theorem]{Proposition}

\newtheorem{remark}[theorem]{Remark}
\newtheorem{problem}[theorem]{Problem}
\newtheorem{subclaim}[theorem]{Subclaim}

\newenvironment{claimproof}[1][\proofname]{\proof[#1]}{\endproof}
\newenvironment{subclaimproof}[1][\proofname]{\proof[#1]}{\endproof}

\theoremstyle{definition}
\newtheorem{definition}[theorem]{Definition}
\newtheorem*{definition*}{Definition}

\author
{
Ema Skottova
}

\thanks{E.S.: Department of Mathematics, ETH Z\"{u}rich, Switzerland.}
\thanks{\texttt{eskottova@ethz.ch}}

\author
{
Raphael Steiner
}
\thanks{R.S.: Department of Mathematics, ETH Z\"{u}rich, Switzerland.}
\thanks{\texttt{raphaelmario.steiner@math.ethz.ch}.
\thanks{The research of R.S. was funded by the SNSF Ambizione Grant No. 216071 of the Swiss National Science Foundation.}
}

\date{\today}

\title{Critical edge sets in vertex-critical graphs}

\begin{document}
\maketitle

\begin{abstract}
\emph{Criticality} is a fundamental notion in graph theory that has been studied continually since its introduction in the early 50s by Dirac. A graph is called \emph{$k$-vertex-critical ($k$-edge-critical)} if it has chromatic number $k$ but the removal of any vertex (edge) lowers the chromatic number to $k-1$. A set of edges in a graph is called \emph{critical} if its removal reduces the chromatic number of the graph.

In 1970, Dirac conjectured a rather strong distinction between the notions of vertex- and edge-criticality, namely that for every $k\ge 4$ there exists a $k$-vertex-critical graph that does not have any critical edges. This conjecture was proved for $k\ge 5$ by Jensen in 2002 and remains open only for $k=4$. A much stronger version of Dirac's conjecture was proposed by Erd\H{o}s in 1985: Let $k\ge 4$ be fixed, and let $f_k(n)$ denote the largest integer such that there exists a $k$-vertex-critical graph of order $n$ in which no set of at most $f_k(n)$ edges is critical. Is it true that $f_k(n)\rightarrow \infty$ for $n\rightarrow \infty$?

Strengthening previous partial results, we solve this problem affirmatively for all $k>4$, proving that $$f_k(n)=\Omega(n^{1/3}).$$ This leaves only the case $k=4$ open. We also show that the stronger lower bound $f_k(n)\ge \Omega(n^{1/2})$ holds along an infinite sequence of numbers $n$. Finally, we provide a first non-trivial upper bound on the functions $f_k$ by proving that $$f_k(n)=O\left(\frac{n}{(\log n)^{\Omega(1)}}\right)$$ for every $k\ge 4$. Our proof of the lower bound on $f_k(n)$ involves an intricate analysis of the structure of proper colorings of a modification of an earlier construction due to Jensen, combined with a gluing operation that creates new vertex-critical graphs without small critical edge sets from given such graphs. The upper bound is obtained using a variant of Szemer\'{e}di's regularity lemma due to Conlon and Fox. \end{abstract}

\section{Introduction}
The chromatic number $\chi(G)$ of a graph $G$ is one of the most fundamental and well-studied graph parameters with applications in combinatorics, scheduling and optimization, whose study dates back at least to 1852 when Francis Guthrie posed his nowadays famous \emph{Four-Color-Problem}. Despite tremendous amounts of effort put into understanding this parameter by graph theorists throughout the last century, many basic questions and conjectures about its properties still remain open today. We refer to the recent surveys of Norin and Scott~\cite{norin,scott} for summaries of some selected seminal work and open problems on the chromatic number. 

In many of these open problems, one is given a particular graph class $\mathcal{G}$ and aims to prove that the chromatic number of all graphs in $\mathcal{G}$ is less than some constant 
$k\in \mathbb{N}$. Typically, in these problems the class $\mathcal{G}$ is \emph{hereditary}, i.e., closed under taking induced subgraphs. In such cases, a powerful general approach towards proving the desired upper bound for the chromatic number of graphs in $\mathcal{G}$ is to go by contradiction: Consider a supposed \emph{smallest counterexample} $G\in \mathcal{G}$, i.e. a smallest graph in the class satisfying $\chi(G)\ge k$, and prove so many properties about this smallest counterexample that eventually a contradiction is reached, concluding the proof. One easily checks that every such smallest counterexample $G$ must have the property that $\chi(G)=k$, but for every vertex $v$ of $G$, the graph $G-v$ obtained from $G$ by removing $v$ results in a $(k-1)$-colorable graph. If the graph class $\mathcal{G}$ is additionally closed under taking arbitrary subgraphs, then one even knows that any graph obtained from $G$ by removing an edge is $(k-1)$-colorable. Motivated by this, the notion of \emph{critical} graphs has emerged and has been continually studied in graph theory since the early 1950s (see~\cite{dirac} for one of the first papers on critical graphs and~\cite{jensentoft,kostochka,stiebitz} for surveys and books on the subject), with applications beyond pure graph coloring problems such as to Tur\'{a}n numbers, see~\cite{simonovits} for an example. 

A vertex, edge or set of edges in a graph $G$ is called \emph{critical} if its removal from $G$ creates a graph with smaller chromatic number than $G$. Given an integer $k\ge 1$, we say that a graph $G$ with $\chi(G)=k$ is \emph{$k$-vertex-critical} if every vertex of it is critical, and we say that it is \emph{$k$-edge-critical} if every edge is critical. Finally, it is called \emph{$k$-critical} if it is both $k$-vertex- and $k$-edge-critical. Several fundamental problems concerning these notions remain open, one of which is the degree to which they relate to each other. While it is easy to see that every $k$-edge-critical graph without isolated vertices is also $k$-vertex-critical, it is not too difficult to construct examples where the converse does not hold. In 1970, Dirac (cf.~\cite{lattanzio,erdos}) put forth a striking conjecture that claims an even stronger distinction between the notions: He conjectured that for every integer $k\ge 4$ there exists a graph that is $k$-vertex-critical, but at the same time \emph{none} of its edges is critical, i.e., removing any edge from the graph does not alter the chromatic number. We remark that the condition $k\ge 4$ in Dirac's conjecture is necessary: For $k=3$ one can easily show that the only $3$-vertex-critical graphs are the odd cycles, and these also have critical edges (in fact, they are also edge-critical). For a long time, not even a single vertex-critical graph was known that does not have critical edges. The first construction of such a graph, resolving the $k=5$ case of Dirac's conjecture, was provided by Brown~\cite{brown} in 1992. It took $10$ more years for further cases of Dirac's conjecture to be resolved: First, Lattanzio~\cite{lattanzio} proved an affirmative answer to Dirac's conjecture for all integers $k\ge 4$ such that $k-1$ is not a prime number. Independently, Jensen~\cite{jensen} managed to prove the conjecture for all $k\ge 5$. Jensen's result remains the state of the art on Dirac's conjecture, notably with the intriguing case $k=4$ still left open after more than 50 years. The only partial result addressing this case in the literature is due to Jensen and Siggers~\cite{jensensiggers}, who proved for every $\varepsilon>0$ the existence of a $4$-vertex-critical graph in which less than an $\varepsilon$-fraction of the edges are critical.

A much stronger version of Dirac's conjecture was discussed in an open problem by Erd\H{o}s~\cite{erdos} in 1985, and has been reiterated in several books and problem collections later (see for instance \cite{chungraham,jensentoft,bloom,chungrahamonline}): Does there, for every $k\ge 4$ and $r\ge 1$, exist a $k$-vertex-critical graph $G$ such that no set of at most $r$ edges in $G$ is critical? In the following, let us call a graph with this property a \emph{$(k,r)$-graph}. The case $r=1$ corresponds to Dirac's conjecture, but apart from the previously mentioned results on the latter not much was known on this problem by Erd\H{o}s. Jensen~\cite{jensen} showed that for every $k\ge 5, r\ge 1$ there exists a $k$-vertex-critical graph in which no set of edges of size at most $r$, all incident to the same vertex, is critical. This however still does not solve the problem for any $r\ge 2$. Rather recently, using a probabilistic argument based on the solution of Shamir's problem~\cite{kahn,parkpham} on matchings in random hypergraphs, Martinsson and the second author~\cite{martinsson} provided a partial solution to Erd\H{o}s's problem by proving that for every fixed $r\ge 1$ and every sufficiently large integer $k$ (at least some quickly growing function of $r$), there indeed exists a $(k,r)$-graph. However, this result does not say anything about the case when $k$ is fixed and $r$ tends to infinity.

Erd\H{o}s also suggested a yet stronger version of his problem, as follows. 

\begin{problem}[cf.~Problem on page 66 of~\cite{chungraham} and~\cite{chungrahamonline, erdos}]\label{prob2}
Let $k\ge 4$. For every $n\in \mathbb{N}$ let $f_k(n)$ denote the largest integer $r\ge 1$ such that there exists a $(k,r)$-graph on $n$ vertices ($f_k(n):=0$ if no such graph exists for any $r\ge 1$). Is it true that $f_k(n)\rightarrow \infty$ for $n\rightarrow \infty$?
\end{problem}
In case of a positive answer, Erd\H{o}s further suggested estimating the asymptotic growth of $f_k(n)$. 

\medskip

\paragraph*{\textbf{Our results.}} In this work, we solve the aforementioned problems of Erd\H{o}s for all $k\ge 5$, leaving only $k=4$ open. 

Concretely, we show that $(k,r)$-graphs exist for every $k\ge 5$, $r\ge 1$. This result significantly strengthens all previous results on the problem, including both Jensen's~\cite{jensen} solution of the cases $(k,1)$ for $k\ge 5$ as well as the aforementioned result by Martinsson and the second author~\cite{martinsson}. Moreover, we obtain the following stronger quantitative result, which in particular shows that $f_k(n)\rightarrow \infty$ for every $k\ge 5$ and thus resolves also Problem~\ref{prob2} with the exception of $k=4$. 

\begin{theorem}\label{thm:lowerbound}
For every fixed $k\ge 5$ we have
$$f_k(n)=\Omega(n^{1/3}).$$
\end{theorem}

In fact, we obtain a better lower bound of order $\sqrt{n}$ along an infinite sequence of numbers $n$, as can be derived from the following, more specific, result by setting $n:=8(k-1)(18r+3)(6r+3)+1$.

\begin{theorem}\label{main-thm-short}
    Let $k\ge 5$, $r \geq 1$ and $n$ be integers such that $n\ge 8(k-1)(18r+3)(6r+3)+1$ and $n\equiv 1 \pmod{8(k-1)(18r+3)}$. Then there exists a $(k,r)$-graph of order $n$.
\end{theorem}
To derive Theorem~\ref{thm:lowerbound} from Theorem~\ref{main-thm-short}, one needs to find a way to create $(k,r)$-graphs of any given (sufficiently large) order $n$. This is done using an additional bootstrapping argument, making use of a graph gluing operation that preserves the property of being a $(k,r)$-graph while allowing to manipulate the order of the graph.

As for upper bounds, it is trivial to see that $f_k(n)< \frac{
n-1}{k-1}$ for every $k\ge 2$ and $n\in \mathbb{N}$: Indeed, let $G$ be any $k$-vertex-critical graph of order $n$. Let $v\in V(G)$ be arbitrarily chosen. Now consider any proper $(k-1)$-coloring of the graph $G-v$. One of the $(k-1)$ colors used in this coloring appears at most $\frac{n-1}{k-1}$ times in $G-v$, and removing any edges from $v$ to a vertex of this color from $G$ results in a $(k-1)$-colorable graph. Our next and final result provides a first non-trivial asymptotic improvement of the upper bound: for every $k\ge 4$ the function $f_k(n)$ grows sublinearly in $n$.
\begin{theorem}\label{thm:upperbound}
There exists an absolute constant $c>0$ such that for every $k\ge 4$ we have
$$f_k(n)=O\left(\frac{n}{(\log n)^c}\right).$$
\end{theorem}
Beyond attacking the remaining open case $k=4$ of Dirac's and Erd\H{o}s's problems, it would also be interesting to further improve our bounds provided in Theorems~\ref{thm:lowerbound} and~\ref{thm:upperbound} and determine the asymptotic growth of $f_k(n)$ in terms of $n$ precisely.

\medskip

\paragraph*{\textbf{Notation.}} By $V(G)$, $E(G)$ we denote the vertex- and edge-set of a graph $G$, respectively. For a vertex $v\in V(G)$, we denote by $N_G(v)$ the set of neighbors of $v$ in $G$.  
Given a subset of vertices $X\subseteq V(G)$, we denote by $G[X]$ the induced subgraph of $G$ with vertex-set $X$ and denote $G-X:=G[V(G)\setminus X]$. We abbreviate $G-v:=G-\{v\}$. Similarly, for an edge-subset $R\subseteq E(G)$ we denote by $G-R$ the spanning subgraph obtained by omitting the edges in $R$. A \emph{proper} coloring of a graph $G$ is a mapping $c:V(G)\rightarrow S$ where $S$ is some finite set of colors, such that $c(u)\neq c(v)$ for all $uv\in E(G)$. We denote by $\log(\cdot)$ and $\exp(\cdot)$ the \emph{natural} logarithm and the exponential function, respectively. With a slight abuse of notation, given two integers $a,b\in \mathbb{Z}$ with $a\le b$ we use $[a,b]:=\{a,\ldots,b\}$ to denote the set of integers from $a$ to $b$. For a natural number $N$ and any two elements $a, b \in [0,N-1]$, we denote by $d_N(a, b) := \min\{|b-a|, N - |b-a|\}$ the \emph{cyclic distance} of $a$ and $b$. Let $A, B \subseteq \mathbb Z$ and $c \in \mathbb Z$. We will use the following notation for sums of sets and elements:
    $A + c=c+A := \{a + c \mid a \in A\}$,
    $cA := \{ca \mid a \in A\}$,
    $A + B := \{a + b \mid a \in A, b\in B\}$.

\medskip

\paragraph*{\textbf{Organization.}} The main technical work for the proof of our main result, namely Theorem~\ref{thm:lowerbound}, lies in establishing Theorem~\ref{main-thm-short}. Since the proof of the latter is detail-heavy and spans several pages, we decided to start off in Section~\ref{sec:deduct} by first presenting the relatively compact proof of Theorem~\ref{thm:lowerbound}, assuming Theorem~\ref{main-thm-short}. In Section~\ref{sec:upper} we then move on by giving the self-contained and also relatively short proof of our sublinear upper bound on $f_k(n)$ in Theorem~\ref{thm:upperbound}. Somewhat unusually for this area of graph coloring, our argument here relies on a variant of Szemer\'{e}di's regularity lemma. Finally, in Section~\ref{sec:main} we provide the proof of our main technical result, Theorem~\ref{main-thm-short}. Our proof considers a $3$-parameter family of circulant graphs $G_{k,m,q}$ which form a modification of a construction previously used by Jensen~\cite{jensen}. Our goal is to show that these graphs (with appropriate choice of the parameters) satisfy the properties required by Theorem~\ref{main-thm-short}. The argument for this needs two steps: \begin{enumerate} \item[(1)] A proof of the $(k-1)$-colorability of $G_{k,m,q}-v$ for all vertices $v$ (captured in Lemma~\ref{colorability-lemma}). This is a relatively simple (yet technical) exercise which is deferred to an appendix (Appendix~\ref{sec:vertcrit}) to not obstruct the flow of reading. \item[(2)] The main work is to prove the non-existence of small critical sets of edges in $G_{k,m,q}$. This proof requires an intricate analysis of the structure of proper $(k-1)$-colorings of graphs of the form $G_{k,m,q}-R$ where $R$ is some bounded-size edge subset, with the goal of eventually leading the assumption of their existence to a contradiction. This itself requires several intermediate steps that work towards establishing a certain periodicity (Lemma~\ref{partial-periodicity-lemma}) of the proper colorings. The latter then easily implies the desired contradiction. 

We remark that the proof of one of the key claims in the proof (namely, Claim~\ref{equiv-structure}) turns out to be significantly more involved in the cases $k\in \{6,8\}$, while a simpler and much shorter argument suffices for all other cases, when $k\notin\{6,8\}$. We thus decided to present this simpler argument in the main proof and to delay the more technical argument for the cases $k\in \{6,8\}$ to an appendix (Appendix~\ref{sec:68}) to not hinder the flow of reading for the main proof.
\end{enumerate}
Finally, we conclude the paper in Section~\ref{sec:conc} by briefly discussing the remaining open case $k=4$ of the problem.

\section{Proof of Theorem~\ref{thm:lowerbound}, assuming Theorem~\ref{main-thm-short}}\label{sec:deduct}
In this section, we show how Theorem~\ref{thm:lowerbound} can be deduced from Theorem~\ref{main-thm-short}. The main missing part of Theorem~\ref{thm:lowerbound} is that we need to find $(k,r)$-graphs for large values of $r$ and any given sufficiently large order $n$, while Theorem~\ref{main-thm-short} only creates such graphs whose order satisfies a rather specific congruence constraint. To bridge this gap, we prove the following lemma, which can be used to create new $(k,r)$-graphs from a given set of $(k,r)$-graphs, thus allowing us to manipulate the order of such graphs. The proof idea is partly inspired by previous work of Wang~\cite{wang}.

\begin{lemma}\label{lem:glue}
Let $k\ge 4$, $r\ge 1$ and $n_1,\ldots,n_t\ge 1, h\ge 1$ be integers such that the following hold:
\begin{itemize}
    \item For every $i\in [1,t]$ there exists a $(k,r)$-graph of order $n_i$, and
    \item there exists a $k$-critical graph of order $h$ and with $t$ edges.
\end{itemize}
Then there also exists a $(k,r)$-graph of order
$$h+\sum_{i=1}^{t}(n_i-1).$$
\end{lemma}
\begin{proof}
Let $G_1,\ldots,G_t$ be $(k,r)$-graphs with $|V(G_i)|=n_i$ for all $i\in [1,t]$ and let $H$ be a $k$-critical graph with $|V(H)|=h$ and $|E(H)|=t$. For each $i\in [1,t]$, let us pick (arbitrarily) some vertex $v_i\in V(G_i)$. Since $G_i$ is a $(k,r)$-graph, there exists a proper $(k-1)$-coloring $c_i:V(G_i)\setminus \{v_i\}\rightarrow [1,k-1]$ of the graph $G_i-v_i$ for every $i\in [1,t]$. For each $i\in [1,t]$ let $A_i, B_i$ be a partition of the neighborhood $N_{G_i}(v_i)$ of $v_i$ in $G_i$ defined as follows: $A_i:=c_i^{-1}(\{1\})\cap N_{G_i}(v_i)$ and $B_i:=c_i^{-1}([2,k-1])\cap N_{G_i}(v_i)$. 

Let us fix (arbitrarily) some enumeration $e_1=u_1w_1,\ldots,e_t=u_tw_t$ of the edges of $H$. We now create a new graph $G$ as follows: We start from the disjoint union of the graphs $H, G_1-v_1,\ldots, G_t-v_t$. We then remove all edges of $H$ and add the following new edges: For every $i\in [1,t]$, we add all possible edges from $u_i$ to $A_i\subseteq V(G_i)$ and from $w_i$ to $B_i\subseteq V(G_i)$. 

We now claim that $G$ is a $(k,r)$-graph of order $h+\sum_{i=1}^{t-1}(n_i-1)$. It is easily checked by construction that the number of vertices of $G$ indeed coincides with the above number, so it remains to show that $G$ is $k$-vertex-critical and that no set of at most $r$ edges of $G$ is critical. 

First, let us show that no set of at most $r$ edges of $G$ is critical. So let $R\subseteq E(G)$ be arbitrary such that $|R|\le r$ and let us show that $\chi(G-R)\ge k$. Towards a contradiction, suppose that there exists a proper $(k-1)$-coloring $c:V(G)\rightarrow [1,k-1]$ of $G-R$. Since $\chi(H)=k$, the restriction $c|_{V(H)}$ cannot form a proper coloring of $H$, and thus there exists some $i\in [1,t]$ such that $c(u_i)=c(w_i)$. Let us now consider the coloring $\phi:V(G_i)\rightarrow [1,k-1]$ of $G_i$, defined as $\phi(x):=c(x)$ for $x\in V(G_i)\setminus \{v_i\}$ and $\phi(v_i):=c(u_i)=c(w_i)$. Since $A_i\cup B_i=N_{G_i}(v_i)$ and $c$ is a proper coloring of $G-R$, one now easily checks that $\phi$ is a proper $(k-1)$-coloring of $G_i-R_i$, where $R_i:=(E(G_i-v)\cap R)\cup \{v_ix|u_ix\in R \text{ or }w_ix\in R\}$. Since clearly $|R_i|\le |R|\le r$, this contradicts our assumption that $G_i$ is a $(k,r)$-graph. This concludes the argument, showing that $G$ indeed does not have any critical edge sets of size at most $r$. 

It remains to show that $G-v$ is $(k-1)$-colorable for every $v\in V(G)$ (which then also ensures $\chi(G)=k$, so $G$ is $k$-vertex-critical). Before beginning the argument, it will turn out useful to make a simple but important observation: For every $i\in [1,t]$ and every two distinct colors $a,b\in [1,k-1]$, there exists a proper $(k-1)$-coloring $\phi_i^{a,b}:(V(G_i)\setminus \{v_i\})\cup \{u_i,w_i\}\rightarrow [1,k-1]$ of the induced subgraph $G_i':=G[(V(G_i)\setminus \{v_i\})\cup \{u_i,w_i\}]$ such that $\phi_i^{a,b}(u_i)=a, \phi_i^{a,b}(w_i)=b$. Indeed, in the case $a=2$ and $b=1$ this directly follows by extending the proper coloring $c_i$ of $G_i-v_i$ with colors $2$ and $1$ at $u_i$ and $w_i$. Since $N_{G_i'}(u_i)=A_i\subseteq c_i^{-1}(\{1\})$ and $N_{G_i'}(w_i)=B_i\subseteq c_i^{-1}([2,k])$, this extension indeed forms a proper coloring of $G_i'$. After appropriate permutation of the colors mapping $2$ and $1$ to $a$ and $b$, respectively, we obtain the desired statement claimed above.

Let us now consider any $v\in V(G)$ and let us prove that $G-v$ is $(k-1)$-colorable. We distinguish two cases depending on the location of $v$:

Suppose first that $v\in V(H)$. Since $H$ is a $k$-critical graph, $H-v$ admits a proper $(k-1)$-coloring $c_H:V(H)\setminus \{v\}\rightarrow [1,k-1]$. We can extend this to a proper coloring of all of $G-v$ as follows: For every $i\in [1,t]$, if $e_i$ is not incident with $v$, then color all vertices in $V(G_i)\setminus \{v_i\}$ as in the proper coloring $\phi_i^{c_H(u_i),c_H(w_i)}$ of $G_i'$. If $e_i$ is incident with $v$, w.l.o.g. say $w_i=v$, then we select some color $s\in [1,k-1]\setminus c_H(u_i)$ and then color all vertices in $V(G_i)\setminus \{v_i\}$ as in the proper coloring $\phi_i^{c_H(u_i),s}$ of $G_i'$. One easily checks that this extension of $c_H$ to $V(G)\setminus \{v\}$ forms a proper $(k-1)$-coloring of $G-v$, as desired.

Next, consider the case that $v\notin V(H)$, i.e. $v\in V(G_j)\setminus \{v_j\}$ for some $j\in [1,t]$. Since $H$ is $k$-critical, there exists a proper $(k-1)$-coloring $c_H:V(H)\rightarrow [1,k-1]$ of the graph $H-e_j$. Since $H$ is $k$-chromatic, we furthermore know that $c_H(u_j)=c_H(w_j)$. Let $\phi_j:V(G_j)\setminus \{v\}\rightarrow [1,k-1]$ be a proper $(k-1)$-coloring of $G_j-v$. Possibly after permuting colors, we may w.l.o.g. assume that $\phi_j(v_j)=c_H(u_j)=c_H(w_j)$. We can now extend $c_H$ to a proper coloring of $G-v$ as follows: We color all vertices in $V(G_j)\setminus \{v_j,v\}$ as in the proper coloring $\phi_j$ of $G_j-v$. For all $i\in [1,k-1]\setminus \{j\}$, we color all vertices in $V(G_i)\setminus \{v_i\}$ as in the proper coloring $\phi_i^{c_H(u_i),c_H(w_i)}$ of $G_i'$. Again, it is not hard to check by the properties of these partial colorings that this extension of $c_H$ to $V(G)\setminus \{v\}$ forms a proper $(k-1)$-coloring of $G-v$, as desired. This concludes the proof that $G$ is $k$-vertex-critical. Together with what we have shown previously this implies that $G$ is a $(k,r)$-graph of the desired order, concluding the proof of the lemma.  
\end{proof}

The lemma above allows to create new $(k,r)$-graphs with certain orders. To be able to create such graphs of order equal to any sufficiently large integer, we will combine Theorem~\ref{main-thm-short} with the following simple statement about the existence of sparse $k$-critical graphs of any given (sufficiently large) order. 

\begin{lemma}\label{lem:critorder}
Let $k\ge 4$. For every integer $n\ge k+3$, there exists a $k$-critical graph of order $n$ with at most $(k-2)n$ edges. 
\end{lemma}
\begin{proof}
We start by proving the special case $k=4$, and later observe that the general case reduces to this case. 

So let $n\ge 7$ be any integer. If $n$ is even, then the wheel graph of order $n$ (that is, the graph obtained from the odd cycle of length $n-1$ by adding a new vertex connected to all vertices on the cycle) is easily checked to be a $4$-critical graph and has $2(n-1)<(4-2)n$ edges, as desired. In fact, the assumption $n\ge 7$ is not required in this case and $n\ge 4$ is sufficient. 

Next, suppose that $n\ge 7$ is odd. We now consider two graphs $G_1$ and $G_2$: The graph $G_1$ is a wheel of even order $n-3\ge 4$, which, as discussed above, is a $4$-critical graph. The graph $G_2$ is simply a $K_4$, which is clearly also $4$-critical. To obtain the desired $4$-critical graph of order $n$, we now perform a so-called \emph{Haj\'{o}s join} of the graphs $G_1$ and $G_2$ (first introduced by Haj\'{o}s~\cite{hajos}): We select an edge $e_1=u_1v_1$ in $G_1$ and an edge $e_2=u_2v_2$ in~$G_2$. We then start from the disjoint union of $G_1$ and $G_2$, delete the edges $e_1$ and $e_2$ from $G_1$ and $G_2$, respectively, next identify the vertices $u_1$ and $u_2$ and finally add a new edge between the vertices $v_1$ and $v_2$. It is well-known and easy to show (cf. Theorem 4.2(c) of~\cite{stiebitz}) that the Haj\'{o}s join of two $k$-critical graphs is always $k$-critical, so the above operation indeed results in a $4$-critical graph of order $|V(G_1)|+|V(G_2)|-1=(n-3)+4-1=n$ and with $|E(G_1)|+|E(G_2)|-1=2(n-4)+6-1<(4-2)n$ edges, as desired. This concludes the proof in the case $k=4$. 

Now consider any integers $k\ge 5$ and $n\ge k+3$. Then $n-(k-4)\ge 7$, and thus by what we just argued there exists a $4$-critical graph $G$ of order $n-(k-4)$ with at most $2(n-(k-4))$ edges. Now consider a graph obtained from $G$ by repeating the following operation $k-4$ times: Add a new universal vertex, i.e. a new vertex made adjacent to all previous vertices. Since performing one step of this operation increases the chromatic number and order of the graph by one, preserves criticality and adds less than $n$ new edges, the eventually obtained graph has order $(n-(k-4))+(k-4)=n$, chromatic number $4+(k-4)=k$, is $k$-critical and has less than $2(n-(k-4))+(k-4)n<(k-2)n$ edges, as desired. This concludes the proof of the lemma in the general case.
\end{proof}

Combining the previous two lemmas and assuming Theorem~\ref{main-thm-short}, we can obtain the following statement.

\begin{theorem}\label{thm:precise}
Let $k\ge 5$ and $r\ge 1$ be integers. Then for every integer $$n\ge 8(k-1)(18r+3)(16(k-2)(k-1)(18r+3)(6r+3)+2)=\Theta(k^3r^3)$$ there exists a $(k,r)$-graph of order $n$. 
\end{theorem}
\begin{proof}[Proof, assuming Theorem~\ref{main-thm-short}]
    Let $n\ge 8(k-1)(18r+3)(16(k-2)(k-1)(18r+3)(6r+3)+2)$ be given arbitrarily. We perform a Euclidean division of $n$ by $8(k-1)(18r+3)$, allowing us to write $n$ as 
    $$n=8(k-1)(18r+3)x+h,$$ where $x$ is a non-negative integer and $8(k-1)(18r+3)\le h<16(k-1)(18r+3)$ (here we used that $n\ge 8(k-1)(18r+3)$).

    Since $h\ge k+3$, Lemma~\ref{lem:critorder} implies that there exists a $k$-critical graph $H$ of order $h$ and with at most $(k-2)h$ edges. In the following, let $t$ denote the number of edges of $H$. Our assumed lower bound on $n$ guarantees that $$x>\frac{n}{8(k-1)(18r+3)}-2\ge 16(k-2)(k-1)(18r+3)(6r+3)> (6r+3)(k-2)h\ge (6r+3)t.$$

    This implies that there exist integers $x_1,\ldots,x_t$ such that $x=x_1+\cdots+x_t$ and $x_i\ge 6r+3$ for every $i\in [1,t]$. Set $n_i:=8(k-1)(18r+3)x_i+1$ for every $i\in [1,t]$. Then by Theorem~\ref{main-thm-short}, for every $i\in [1,t]$ there exists a $(k,r)$-graph $G_i$ of order $n_i$. We are now in the position to apply Lemma~\ref{lem:glue} to the numbers $n_1,\ldots,n_t$ and $h$. We thus obtain the existence of a $(k,r)$-graph of order
    $$h+\sum_{i=1}^{t}(n_i-1)=h+8(k-1)(18r+3)\sum_{i=1}^{t}x_i=n.$$ This concludes the proof of the theorem.
\end{proof}
Finally, observe that Theorem~\ref{thm:precise} directly implies the statement of Theorem~\ref{thm:lowerbound}: For every sufficiently large integer $n$, the largest number $r$ respecting the condition of Theorem~\ref{thm:precise} forms a lower bound for $f_k(n)$ and is easily seen to be of order $\Theta(n^{1/3})$. This establishes Theorem~\ref{thm:lowerbound} conditional on assuming Theorem~\ref{main-thm-short}, as desired.
\section{Proof of Theorem~\ref{thm:upperbound}}\label{sec:upper}
In this section, we provide the proof of the sublinear upper bounds on the functions $f_k$ stated in Theorem~\ref{thm:upperbound}. As the main ingredient of the proof we use the following variant of Szemer\'{e}di's regularity lemma due to Conlon and Fox, that compared to the classical regularity lemma guarantees only a weak form of regular partitions, but comes with better quantitative bounds. We are grateful to Yuval Wigderson for pointing us to this variant of the regularity lemma. Using the classical regularity lemma and with the same argument we present below one can also directly obtain a sublinear upper bound $f_k(n)=o(n)$, albeit with a weaker separation of $f_k(n)$ from $n$. 
To state the result by Conlon and Fox, we need to briefly recall the notion of \emph{regular pairs} that is central to all regularity lemmas. Recall that given a graph $G$ and two non-empty subsets $A, B\subseteq V(G)$ of vertices, the \emph{density} $d(A,B)$ is defined as the ratio 
$$\frac{|\{(a,b)\in A\times B|ab\in E(G)\}|}{|A||B|}.$$

\begin{definition}
    Let $G$ be a graph and $\varepsilon>0$ a constant. We say that a pair $(X,Y)$ of non-empty subsets of $V(G)$ is \emph{$\varepsilon$-regular} if there exists some $\alpha\in \mathbb{R}$ such that for all subsets $A\subseteq X, B\subseteq Y$ with $|A|\ge \varepsilon |X|, |B|\ge \varepsilon |Y|$ we have $\alpha<d(A,B)<\alpha+\varepsilon$.
\end{definition}
We can now state the aforementioned result due to Conlon and Fox.
\begin{lemma}[Conlon and Fox, Lemma~8.1 in~\cite{conlonfox}]\label{lem:conlonfox}
There exists an absolute constant $C>0$ such that the following holds for each $0 <\varepsilon< \frac{1}{2}$: For every graph $G$ there is an equitable partition\footnote{Recall that a partition is called equitable if any two sets of the partition differ in size by at most one.}
$V(G) = V_1 \cup \cdots \cup V_t$ of the vertex-set into less than $k(\varepsilon):=\exp(\varepsilon^{-C})$ parts such that for each $i\in [1,t]$ there is a
partition $V(G) = V_{i,1} \cup \cdots \cup V_{i,{j_i}}$ of the vertex-set with $j_i \le K(\varepsilon):=\frac{C}{\varepsilon}$, such that for all $i\in [1,t]$ and $j\in [1,j_i]$
the pair $(V_i, V_{i,j})$ is $\varepsilon$-regular in $G$.
\end{lemma}
With this tool at hand, we are now ready to present our proof of Theorem~\ref{thm:upperbound}.
\begin{proof}[Proof of Theorem~\ref{thm:upperbound}]
Let $k\ge 4$ be fixed in the following. Set $c:=1/C$, where $C>0$ is the absolute constant from Lemma~\ref{lem:conlonfox}. We will show that for every sufficiently large positive integer $n$, every $k$-vertex-critical graph of order $n$ admits a critical set of edges of size at most $\frac{n}{(\log n)^c}$. This then implies $f_k(n)=O\left(\frac{n}{(\log n)^c}\right)$, as desired. 

So let $n$ be any sufficiently large positive integer (concretely, such that $(\log n)^{-c}<\frac{1}{2(k-1)}$) and let $G$ be any given $k$-vertex-critical graph of order $n$. Let $\varepsilon:=(\log n)^{-c}<\frac{1}{2}$. By applying Lemma~\ref{lem:conlonfox} to $G$, we find that there exists an equitable partition $(V_i)_{i=1}^{t}$ of $V(G)$ into $t< k(\varepsilon)=\exp(\varepsilon^{-C})=n$ parts (in particular, this implies $V_i\neq \emptyset$ for all $i\in [1,t]$), and for each $i\in [1,t]$ a partition $(V_{i,j})_{j=1}^{j_i}$ of $V(G)$ into at most $K(\varepsilon)=\frac{C}{\varepsilon}=C(\log n)^c$ parts, such that for all $i\in [1,t]$ and $j\in [1,j_i]$ the pair $(V_i,V_{i_j})$ is $\varepsilon$-regular in $G$. 

W.l.o.g. let $V_1$ be a largest set in the equitable partition $(V_i)_{i=1}^{t}$ of $V(G)$, such that $|V_1|=\left\lceil \frac{n}{t}\right\rceil \ge 2$. By vertex-criticality of $G$, for every vertex $v\in V_1$ there exists a proper $(k-1)$-coloring $c_v:V(G)\setminus \{v\}\rightarrow [1,k-1]$ of $G-v$. For each $v\in V(G)$, there exists some color in $[1,k-1]$ that appears in the coloring $c_v$ at least $\frac{|V_1|-1}{k-1}$ times on the set $V_1\setminus \{v\}$. Possibly after permuting the colors of the coloring $c_v$ (for each $v$ individually), we may w.l.o.g assume in the following that color $1$ has this property, for every $v\in V_1$. 

Let us now consider choosing a uniformly random vertex $v\in V_1$. Our goal in the following will be to estimate the expectation of the random variable $X$, defined as the number of neighbors of $v$ in $G$ that are assigned color $1$ by $c_v$. Note that we can write $$X=\sum_{j=1}^{j_1}X_j,$$ where $X_j$ counts the number of neighbors of $v$ in $V_{1,j}$ that are assigned color $1$ by $c_v$. 

Let us now consider any $j\in [1,j_1]$ and let us upper bound the expectation of $X_j$. To do so, let us first consider the case that $d(V_1,V_{1,j})\le \varepsilon$. In this case, we have
$$\mathbb{E}[X_j]\le \mathbb{E}[|N_G(v)\cap V_{1,j}|]=\frac{1}{|V_1|}\sum_{v\in V_1}|N_G(v)\cap V_{1,j}|$$
$$=|V_{1,j}|\frac{|\{(v,w)\in V_1\times V_{1,j}|vw\in E(G)\}|}{|V_1||V_{1,j}|}=d(V_1,V_{1,j})|V_{1,j}|\le \varepsilon |V_{1,j}|.$$

Next, suppose that $d(V_1, V_{1,j})>\varepsilon$. Since the pair $(V_1,V_{1,j})$ is $\varepsilon$-regular by choice of the partitions, we then have that there exists some $\alpha\in \mathbb{R}$ such that $\alpha<d(A,B)<\alpha+\varepsilon$ for all subsets $A\subseteq V_1, B\subseteq V_{1,j}$ with $|A|\ge \varepsilon |V_1|, |B|\ge \varepsilon |V_{1,j}|$. Applying this inequality to the sets $A=V_1, B=V_{1,j}$ and using our assumption above, it follows that $\alpha+\varepsilon>d(V_1, V_{1,j})>\varepsilon$ and thus $\alpha>0$. We now claim that we must have $|c_u^{-1}(\{1\})\cap V_{1,j}|< \varepsilon |V_{1,j}|$ for every $u\in V_1$. Indeed, suppose that for some $u\in V_1$ we had $|c_u^{-1}(\{1\})\cap V_{1,j}|\ge \varepsilon |V_{1,j}|$. Now let us set $A:=c_u^{-1}(\{1\})\cap V_1$ and $B:=c_u^{-1}(\{1\})\cap V_{1,j}$. By the assumption we just made, we have $|B|\ge \varepsilon |V_{1,j}|$, and by our above assumption on the coloring $c_u$ we also have $|A|\ge \frac{|V_1|-1}{k-1}\ge \varepsilon |V_1|$, using our assumption that $n$ is sufficiently large (and thus $\varepsilon=(\log n)^{-c}<\frac{1}{2(k-1)})$. Applying the above inequality, we now obtain that $d(A,B)>\alpha>0$. This, however, is a contradiction, since there can be no edges between $A$ and $B$ in $G$, as $c_u$ is a proper coloring. This contradiction shows that our above assumption was wrong, and so we indeed must have $|c_u^{-1}(\{1\})\cap V_{1,j}|< \varepsilon |V_{1,j}|$ for every $u\in V_1$. In particular, it follows that $\mathbb{E}[X_j]\le \mathbb{E}[|c_v^{-1}(\{1\})\cap V_{1,j}|]\le \varepsilon |V_{1,j}|$ also in this second case for $j$. Having shown $\mathbb{E}[X_j]\le \varepsilon |V_{1,j}|$ for all $j\in [1,j_1]$ we now find, using linearity of expectation:

$$\mathbb{E}[X]= \sum_{j=1}^{j_1}\mathbb{E}[X_j]\le \varepsilon \sum_{j=1}^{j_1}{|V_{1,j}|}=\varepsilon n=\frac{n}{(\log n)^c}.$$

By definition of $X$, this implies the existence of some $v\in V_1$ having at most $\frac{n}{(\log n)^c}$ neighbors of color $1$ in the coloring $c_v$ of $G-v$. Let $R$ be the set of all edges in $G$ that connect $v$ to a neighbor of color $1$ in $c_v$. Then one easily checks that $G-R$ admits a proper $(k-1)$-coloring, obtained by extending $c_v$ with color $1$ at $v$. This means that $R$ is a critical set of edges in $G$ of size at most $\frac{n}{(\log n)^c}$, as desired. This concludes the proof of the theorem.
\end{proof}
\section{Proof of Theorem~\ref{main-thm-short}}\label{sec:main}
In this section, we will prove Theorem~\ref{main-thm-short}. For the proof we use a construction of a family of \emph{circulant graphs}. Recall that a graph $G$ on $N$ vertices is called circulant if its vertices can be labeled $v_0, v_1, \dots, v_{N - 1}$ such that there exists a \emph{distance set} $D \subseteq \{1, 2, \dots, \lfloor\frac{N}{2}\rfloor\}$ for which it holds that any vertex $v_i$ is adjacent to another vertex $v_j$ if and only if $d_N(i, j) \in D$. The circulant graphs we consider are similar to those used previously by Jensen~\cite{jensen} to construct $k$-vertex-critical graphs without critical edges for all $k\ge 5$. However, a crucial modification that is required for our proof is the addition of a significant amount of new edges to Jensen's graphs. We are careful to add these edges in such a way that $(k-1)$-colorability of the vertex-deleted subgraphs, and thus $k$-vertex-criticality, is preserved. On the other hand, the additional edges later prove crucial when establishing the absence of small critical edge sets in our graphs. Namely, given any graph $G'$ obtained from our graph $G$ by omitting a bounded number of edges, the additional edges in $G$ enforce more constraints and redundancy on a supposed proper $(k-1)$-coloring of $G'$. This enables us to prove important structural properties of such colorings, eventually leading to a contradiction and showing that $\chi(G')=k$, as desired. Specifically, the main intermediate goal of the proof will be to show that any proper $(k-1)$-coloring of such a graph $G'$ must be periodic along the natural circular order of the vertices with a certain period.

Before going into more details, let us explicitly define the aforementioned $3$-parameter family of circulant graphs we will use in our proof. Before that we define an important parameter $n_{k,m}$ which, for reasons alluded to above and that shall become concrete later, will from now on be referred to as the \emph{period length}. It is defined as:
    \begin{equation}\notag
        n_{k,m} \coloneqq 
        \begin{cases}
        (k - 1)m & \text{if } k \text{ is odd} \\
        2(k - 1)m & \text{if } k \text{ is even}.
        \end{cases}
    \end{equation}

    Given positive integers $k, m, q$, where $q$ is even, we now define a circulant graph $G_{k,m,q}$ of order $N := qn_{k,m} + 1$ via the following distance set: $D_{k, m, q} \coloneqq D_1 \cup D_2 \cup D_3$, where
    \begin{align*}
        D_1 &\coloneqq \{1, 3, \dots, 2m - 1\}\\
        D_2 &\coloneqq \begin{cases}
            \bigcup\limits_{q' = 0}^{\frac{q}{2} - 1} [2m, (k-3)m + 1] + q'n_{k, m} & \text{if } k \text{ is odd}\\
            \bigcup\limits_{q' = 0}^{\frac{q}{2} - 1} [2m, (k - 4)m + 2] + q'n_{k, m} & \text{if } k \text{ is even}
        \end{cases}\\
        D_3 &\coloneqq \begin{cases}
            \emptyset & \text{if } k \text{ is odd}\\
            \bigcup\limits_{q' = 0}^{\frac{q}{2} - 1} [(k + 2)m - 1, (2k - 4)m + 1] + q'n_{k,m} & \text{if } k \text{ is even}.
        \end{cases}
    \end{align*}
    
    As mentioned before, the difference between the graphs $G_{k,m,q}$ defined here and the graphs studied by Jensen~\cite{jensen} lies in the richer distance set used for our construction: While we shift entire intervals in $D_2$ and $D_3$ with multiples of the period length $n_{k,m}$, Jensen~\cite{jensen} instead only shifted the $2$-element sets consisting of the endpoints of these intervals.

Let us now move on with the discussion of the proof of Theorem~\ref{main-thm-short}. Concretely, we will prove the following statement about the circulant graphs defined above, which then easily implies Theorem~\ref{main-thm-short}.
    \begin{theorem}\label{main-thm}
    Let $k , r, q$ and $m$ be integers such that $k\ge 5$, $r\ge 1$, $q \geq 24r + 12$ is a multiple of $4$ and $m > 18r+2$. Then $G_{k, m, q}$ is a $(k,r)$-graph.
    \end{theorem}

    To see that Theorem~\ref{main-thm} implies Theorem~\ref{main-thm-short}, let $k\ge 5, r\ge 1$ be given and let $n\ge 8(k-1)(18r+3)(6r+3)+1$ with $n\equiv 1 \pmod{8(k-1)(18r+3)}$ be an integer. Then we can define $m:=18r+3$ and $q:=\frac{n-1}{n_{k,m}}$. Using the conditions on $n$ and the definition of $n_{k,m}$, one easily verifies that $q\ge 24r+12$ and that $q$ is a multiple of $4$. Thus, by Theorem~\ref{main-thm} we have that $G_{k,m,q}$ is a $(k,r)$-graph of order $qn_{k,m}+1=n$, yielding the statement of Theorem~\ref{main-thm}, as desired.

    So it remains to establish Theorem~\ref{main-thm}. The proof has two parts: (1) showing $G_{k,m,q} - v$ is $(k - 1)$-colorable for every vertex $v$, and (2) that for every set $R$ of at most $r$ edges, $G_{k,m,q}- R$ is not $(k-1)$-colorable. The proof of part (1) is similar to the analogous step of the proof by Jensen~\cite{jensen}, however, due to the additional edges, the $(k-1)$-colorability of the vertex-deleted subgraphs of our graphs is not a formal consequence of Jensen's results, and needs to be additionally checked. This simple, yet technical first step of the proof is deferred to Appendix~\ref{sec:vertcrit} to not hinder the presentation of the main novel contribution in this paper, which is in the second part of the proof when showing non-criticality of the edges. As announced earlier, a key step of this part (2) of the proof will be to show that any proper $(k-1)$-coloring of $G_{k,m,q}-R$ for some edge set $R$ of size at most $r$ is periodic (with period $n_{k,m}$) along the circular ordering of the vertices of $G_{k,m,q}$. The following two lemmas formally summarize the statements proved in steps (1) and (2).

    \begin{lemma}[Colorability]\label{colorability-lemma}
        Let $k, r, m, q$ be as in Theorem~\ref{main-thm}. For every $v \in V(G_{k, m, q})$, the graph $G_{k, m, q} - v$ is $(k-1)$-colorable.
    \end{lemma}

    We will need a few definitions to state the second lemma: Let $G$ be a circulant graph and $G'$ a spanning subgraph of $G$. We call a vertex $v$ \emph{unaffected} (in $G'$) if $N_G(v) = N_{G'}(v)$ and \emph{affected} otherwise. We call a set $S$ of $l$ distinct vertices \emph{consecutive} if $S = \{v_{i_1}, v_{i_2}, \dots, v_{i_l}\}$ where for all $j \in [2, l]$ we have $d_N(i_j, i_{j - 1}) = 1$. Finally, we say that a coloring $\varphi$ is \emph{periodic} (with period $n_{k, m}$) on a set $S$ of vertices if for all $v_i, v_j \in S$ satisfying $d_N(i, j) = n_{k, m}$ we have $\varphi(v_i) = \varphi(v_j)$.

    \begin{lemma}[Partial Periodicity]\label{partial-periodicity-lemma}
        Let $k, r, m, q$ be as in Theorem~\ref{main-thm}. Let $G'=G_{k, m, q}-R$ be a graph obtained from $G_{k, m, q}$ by deleting an edge set $R$ with $|R|\le r$. Let $S$ be a set of at least $3n_{k, m}$ consecutive unaffected vertices in $G'$. Let $S'$ be a set of at most $\frac{N-1}{4}$ consecutive vertices containing $S$. Then every proper $(k - 1)$-coloring of $G'$ is periodic on $S'$.
    \end{lemma}

    We first demonstrate how Theorem~\ref{main-thm} can be deduced from Lemmas~\ref{colorability-lemma} and~\ref{partial-periodicity-lemma} before proceeding with the proof of Lemma~\ref{partial-periodicity-lemma}.

    \begin{proof}[Proof of Theorem~\ref{main-thm}, assuming Lemma~\ref{colorability-lemma} and~\ref{partial-periodicity-lemma}]
        Let $G := G_{k, m, q}$ and $D := D_{k, m, q}$. By Lemma~\ref{colorability-lemma} $G-v$ is $(k-1)$-colorable for every $v$, so in particular $\chi(G) \leq k$. It is left to show that every graph obtained from $G$ by removing at most $r$ edges is not $(k-1)$-colorable. Towards a contradiction, suppose that there exists a graph $G'=G-R$ with $R\subseteq E(G), |R|\le r$ that admits a proper $(k-1)$-coloring $\phi$.

        We start by showing that every set $S'$ of $\frac{N-1}{4}$ consecutive vertices has a subset of $3n_{k,m}$ consecutive unaffected vertices. Then we will use Lemma~\ref{partial-periodicity-lemma} to show that $\varphi$ is periodic. Finally, using the periodicity we will show that $\varphi$ is not proper.

        Since $G$ is symmetric with respect to cyclic shifts, we may assume without loss of generality that $S' = \{v_1, \dots, v_{\frac{N - 1}{4}}\}$ in the following. Let $a$ be the number of affected vertices in $S'$. Then $a \leq 2|R|\le 2r$. We next partition the set of unaffected vertices of $S'$ into $a + 1$ sets $U_0, \dots, U_{a}$ where $v_i \in U_j$ if and only if there are exactly $j$ affected vertices with a smaller index than $i$. Observe that each $U_j$ is a set of consecutive unaffected vertices and that their union has $\frac{N-1}{4} - a$ vertices. Then the largest set $U_i$ satisfies
        \begin{equation}\notag
            |U_i| \geq \frac{\frac{N - 1}{4} - 2r}{2r + 1} = \frac{q}{8r + 4}n_{k, m} - \frac{2r}{2r + 1} > 3n_{k, m} - 1.
        \end{equation}
        This shows that every set $S'$ of $\frac{N-1}{4}$ consecutive vertices contains a subset $S$ of $3n_{k,m}$ consecutive unaffected vertices, as desired. Applying Lemma~\ref{partial-periodicity-lemma}, we can now conclude that $\varphi$ is periodic on every set $S'$ of $\frac{N-1}{4}$ consecutive vertices.

        Now we will use the periodicity to show that $\varphi$ cannot be proper.
        From what we argued above it follows that there is some unaffected vertex in $G'$. By cyclic symmetry, we may w.l.o.g. assume that $v_0$ is unaffected. Then $v_0$ and $v_{N-1}$ are adjacent in $G$ (because $d_N(0, N-1) = 1 \in D_1$) and thus also in $G'$ (as $v_0$ is unaffected). We will now show that they are assigned the same color by $\varphi$, arriving at the desired contradiction. Since $q > 4$, we have $n_{k, m} < \frac{N - 1}{4}$. Hence, for every $i \in [1, q]$ there is a set of $\frac{N-1}{4}$ consecutive vertices containing $v_{(i-1)n_{k,m}}$ and $v_{in_{k,m}}$. Therefore, we can use the periodicity to get

        \begin{equation}\label{contradiction-eqn}
            \varphi(v_0) = \varphi(v_{n_{k,m}}) = \dots = \varphi(v_{qn_{k,m}}).
        \end{equation}

        Recall that $N-1 = qn_{k,m}$, so (\ref{contradiction-eqn}) tells us that $\varphi(v_{0}) = \varphi(v_{N-1})$. This implies that $\varphi$ is not proper, so our assumption was false. Hence, every graph obtained from $G$ by removing at most $r$ edges has chromatic number~$k$. It follows that $G_{k,m,q}$ is a $(k,r)$-graph, concluding the proof of the theorem.
    \end{proof}

As mentioned, the proof of Lemma~\ref{colorability-lemma} can be found in Appendix~\ref{sec:vertcrit}. Thus our goal in the remainder of this section is to prove Lemma~\ref{partial-periodicity-lemma}, i.e. to show that every proper $(k-1)$-coloring of $G'$ is periodic on a consecutive set $S'$ if the latter contains a large enough subset of consecutive unaffected vertices. Before starting the formal proof, we give a brief overview, providing some more intuition behind both the statement and proof of Lemma~\ref{partial-periodicity-lemma}. We hope that this helps the reader's orientation during the rather long proof. 

\medskip

\paragraph*{\textbf{Motivation and overview of the proof.}} To give the proper motivation, we have to start by considering the structure of the proper $(k-1)$-colorings of the vertex-deleted subgraphs $G_{k,m,q}-v$ we use in the proof of Lemma~\ref{colorability-lemma} (see Appendix~\ref{sec:vertcrit} for more details). Note that by circular symmetry of the graphs, it suffices to consider $v=v_0$. For odd values of $k$, the proper $(k-1)$-coloring (called $\phi_J$) of $G_{k,m,q}-v_0$ looks as follows:

\begin{table}[H]
    \centering
\begin{tabular}{c c c c c c c}
    $1$& $2$& $1$& $2$& \dots& $1$& $2$\\
    $3$& $4$& $3$& $4$& \dots& $3$& $4$\\
    \vdots& & & &\vdots & &\vdots \\
    $k-4$& $k-3$& $k-4$& $k-3$& \dots& $k-4$& $k-3$\\
    $k-2$& $k-1$& $k-2$& $k-1$& \dots& $k-2$& $k-1$\\
\end{tabular}
\caption{Table illustrating the proper $(k-1)$-coloring $\varphi_J$ of $G_{k,m,q}-v_0$ for odd $k$ on the first $\frac{k-1}{2}\cdot 2m=(k-1)m=n_{k,m}$ vertices $v_1,\ldots,v_{n_{k,m}}$. Every row has $2m$ elements. For every $i \in \left[0, \frac{k - 1}{2} - 1\right]$ and $j \in [1, 2m]$ the entry in the $i$th row and $j$th column  represents the color of vertex $v_{2mi + j}$. The full coloring $\phi_J$ is obtained by periodically repeating the illustrated color pattern of length $n_{k,m}$ exactly $q$ times, such that all $qn_{k,m}=N-1$ vertices are colored.\label{table-odd-coloring}}
\end{table}

If $k$ is even, then we cannot pair up the colors as nicely as in the odd case above. Instead, we pair up each color with two different colors. Each color will appear in two rows instead of one and the pairs sharing a color are as far from each other as possible:
\begin{table}[H]
    \centering
\begin{tabular}{c c c c c c c}
    $1$& $2$& $1$& $2$& \dots& $1$& $2$\\
    $3$& $4$& $3$& $4$& \dots& $3$& $4$\\
    \vdots& & & &\vdots & &\vdots \\
    $k-3$& $k-2$& $k-3$& $k-2$& \dots& $k-3$& $k-2$\\
    $k-1$& $1$& $k-1$& $1$& \dots& $k-1$& $1$\\
    $2$& $3$& $2$& $3$& \dots& $2$& $3$\\
    \vdots& & & &\vdots & &\vdots \\
    $k-4$& $k-3$& $k-4$& $k-3$& \dots& $k-4$& $k-3$\\
    $k-2$& $k-1$& $k-2$& $k-1$& \dots& $k-2$& $k-1$\\
\end{tabular}
\caption{Table illustrating the $(k-1)$-coloring $\varphi_J$ of $G_{k,m,q}-v_0$  for even $k$ restricted to the first $(k-1)\cdot 2m=n_{k,m}$ vertices $v_1,\ldots,v_{n_{k,m}}$. Every row still has $2m$ elements and there are $(k - 1)$ rows. If we repeat the illustrated color pattern of length $(k-1)\cdot 2m=n_{k,m}$ exactly $q$ times, then we obtain the full coloring $\phi_J$ of all the $N-1=qn_{k,m}$ vertices.\label{table-even-coloring}}
\end{table}

The coloring $\phi_J$ illustrated above is the same coloring that Jensen used in his paper~\cite{jensen}. The main idea of Jensen's proof of non-criticality of small sets of edges, all incident to the same vertex, was as follows (let us denote Jensen's graph in the following by $J$): Let $v$ be a vertex incident to all removed edges. W.l.o.g. (by cyclic symmetry) let $v=v_0$. Jensen then proceeds to show that the $(k-1)$-coloring $\phi_J$ is, up to color permutation, the unique proper $(k-1)$-coloring of the graph $J-v_0$. One can then derive a contradiction by checking that the vertex $v_0$ has too many neighbors in each of the $(k-1)$ colors in this unique coloring of $J-v_0$, so that no matter which incident edges of $v_0$ are removed, a monochromatic edge remains.

In contrast, in the proof of our Lemma~\ref{partial-periodicity-lemma}, we cannot assume any specific structure of the set of edges $R$ that we delete, which could be distributed all over the graph and theoretically break Jensen's color-pattern at various places. Thus, it seems difficult to obtain uniqueness of colorings as in Jensen's proof. 
Instead, to show Lemma~\ref{partial-periodicity-lemma} we take a different approach that circumvents uniqueness altogether and instead establish several structural properties of the $(k-1)$-colorings of the graph $G'=G_{k,m,q}-R$. These are strong enough so that we may eventually use them to conclude periodicity on any long consecutive set $S'$ of vertices containing a sufficiently large unaffected subset $S$, yielding the statement of the lemma. 

While reading the statements and proofs of these structural properties, it shall be helpful for the reader to keep the colorings $\phi_J$ of the graph $G_{k,m,q}-v$ in mind as a motivation: While, as mentioned, we do not prove here that there is a unique $(k-1)$-coloring of $G'$ that extends the coloring $\phi_J$ of $G_{k,m,q}-v_0$ illustrated above, the structural properties we prove about proper $(k-1)$-colorings of $G'$ still resemble features shared by the coloring $\phi_J$ illustrated above. For example, we eventually establish that the color classes in any proper $(k-1)$-coloring of $G'$ are made up of a union of blocks of the form $\{v_i,v_{i+2},\ldots,v_{i+2m-2}\}$, just as in the coloring $\phi_J$.

We will now proceed with the proof of the lemma, with the following roadmap: Starting from any proper $(k-1)$-coloring $\varphi$ of $G'$ we first make some observations about the restriction of $\varphi$ to the unaffected part $S$ of the graph. In particular, that it is periodic on $S$ and has a certain structure resembling the coloring $\varphi_J$ (Claims~\ref{local-periodicity-claim},~\ref{structure-claim},~\ref{equiv-structure}). Next, we will show that the edges incident to $S$ already restrict the color of every vertex in the bigger consecutive portion $S'$ of the graph to at most two options (Claim~\ref{two-colors-claim}). Then, we will show that if many vertices close to each other in the consecutive ordering have the same color, their indices must have the same parity (Claim~\ref{2m-claim}). The latter statement, in combination with the remaining aforementioned properties, is then instrumental for deducing that $\varphi$ must be periodic on $S'$. The argument for this last step can essentially be thought of as an induction along the cyclic ordering of vertices: We consider the first vertex $v_i$ whose assigned color in $\phi$ violates $n_{k,m}$-periodicity, and then derive a contradiction from there.

\begin{proof}[Proof of Lemma~\ref{partial-periodicity-lemma}]
We start by introducing some notation that will be used repeatedly in the following:
Given some $i \in [1, (q-1)n_{k,m} + 1]$, we denote by $H_i = \{v_j \mid j \in [i, i + n_{k,m} - 1]\}$
    the consecutive set of $n_{k,m}$ vertices starting with $v_i$. Furthermore, throughout the proof we will denote by $F:=\{0,2,\ldots,2m-2\}$ the set of $m$ consecutive even integers starting with $0$. We will also denote $G:=G_{k,m,q}$ and $D:=D_{k,m,q}$ in the following.

    We start the proof by restricting the sets $S$ and $S'$ that we have to consider in the statement of the lemma. As both $S$ and $S'$ are sets of consecutive vertices, we can partition $S'$ into $S_1 \cup S \cup S_2$ where the $S_i$ are (potentially empty) sets of consecutive vertices. If we can prove the lemma for all $S''$ such that $S'' \setminus S$ is a set of consecutive vertices, we can then apply it to both $S_1 \cup S$ and $S \cup S_2$ to obtain that it also holds for the original $S'$. Therefore, it is sufficient to consider $S'$ such that $S' \setminus S$ is a set of consecutive vertics.
    
    We also note that cyclic shifts and reversing the indices of the vertices are graph isomorphisms which preserve the cyclic distances of indices. Using a cyclic shift, possibly in conjunction with a reversal, we may assume without loss of generality in the following that $S = \left\{v_1, \dots, v_{3n_{k, m}}\right\}$ and $S' = \left\{v_1, \dots, v_{\frac{N-1}{4}}\right\}$.

    All vertices considered in the remainder of the proof will be in the sets $\left\{v_1, \dots, v_{\frac{N-1}{2}}\right\}$ or $\left\{v_1, \dots, v_{\frac{N-1}{4} + n_{k,m}}\right\}$. Since $\frac{N-1}{4}+n_{k,m}< \frac{N}{2}$ by our assumptions on the parameters, the cyclic distance of any two vertices we consider will always be the absolute value of the difference of their indices.
    
    Let $\varphi:V(G')\rightarrow [1,k-1]$ be any given proper $(k - 1)$-coloring of $G'$. Our goal is to show that $\varphi$ is periodic on $S'$. We start by showing that it is periodic on $S$. The proof of the following claim draws from ideas in Jensen's proof~\cite{jensen} of unique colorability of vertex-deleted subgraphs of his graphs.

    \begin{claim}[Local Periodicity]\label{local-periodicity-claim}
        The coloring $\varphi$ satisfies the following:

        \begin{enumerate}[label=(\alph*)]
            \item\label{lp-first} Every set of $n_{k, m}$ consecutive vertices in $S$ has exactly $\frac{n_{k, m}}{k-1}$ vertices of each color under $\varphi$.
            \item\label{lp-second} The coloring $\varphi$ is periodic on $S$.
        \end{enumerate}
    \end{claim}

    \begin{claimproof}
        We first assume \ref{lp-first} and show \ref{lp-second}: Let $i \in [1, 2n_{k,m}]$ and consider $H_i$ and $H_{i + 1}$, the sets of $n_{k, m}$ consecutive vertices starting at $v_i$ and $v_{i + 1}$ respectively. Their intersection satisfies
        \begin{equation}\label{loc-per-intersection-eqn}
            H_i \cap H_{i+1} = H_i \setminus \{v_i\} = H_{i + 1} \setminus \{v_{i + n_{k,m}}\}.
        \end{equation}
        From \ref{lp-first} we know that $\varphi(v_i)$ appears $\frac{n_{k,m}}{k - 1}$ times in $H_i$ and in $H_{i+1}$. From~(\ref{loc-per-intersection-eqn}) we can conclude that only $\frac{n_{k,m}}{k-1} - 1$ vertices in $H_i \cap H_{i+1}$ have color $\varphi(v_i)$. Therefore, $\varphi(v_i) = \varphi(v_{i + n_{k, m}})$. As $i\in [1,2n_{k,m}]$ was arbitrary, this establishes that $\phi$ is periodic on $S=\{v_1,\ldots,v_{3n_{k,m}}\}$.

        It remains to prove \ref{lp-first}. To do so, we consider two cases depending on the parity of $k$:

        \medskip
    
        \paragraph*{\textbf{Case~1.} $k$ is odd.}
     
        Let the set of $n_{k, m}$ consecutive unaffected vertices be $H_i$ for some $i \in [1, 2n_{k,m} + 1]$. Let 
        \begin{equation}\notag
            A := [1, n_{k,m}] \cap (\{i, i + 1\} + 2m\mathbb{Z}).
        \end{equation}
        Intuitively, $A$ contains exactly the first vertex of each color in the coloring $\varphi_J$ if the coloring was shifted to start at $i$. We note that $|A| = k - 1$. For $j \in [1, m]$ we let 
        
        \begin{equation}\notag
            V_j := \{\, v_{a + 2(j - 1)} \mid a \in A \,\}    
        \end{equation}
        and claim that these sets partition $H_i$ into cliques of size $k-1$: 
        
        The sets are cliques: Let $v_x, v_y \in V_j$ be distinct. Then 
        \begin{equation}\notag
            d_N(x, y) = |x - y| \in \{1, 2m - 1\} \cup [2m, (k - 3)m + 1].
        \end{equation} 
        We note that $\{1, 2m - 1\} \subseteq D_1$ and $[2m, (k - 3)m + 1] \subseteq D_2$. Therefore, all distances are in $D$. Since furthermore all vertices of $V_j$ are contained in $S$ and are thus unaffected, it follows that $V_j$ is indeed a clique in $G'$ for every $j \in [1, m]$.

        The sets are a partition of $H_i$: Given any $v_x \in H_i$ let us write $x = i + 2am + b$ where $a \in \Z_{\ge 0}$ and $b \in [0, 2m - 1]$. Then $v_x \in V_{\lceil\frac{b}{2}\rceil}$. As every $V_j$ has exactly $k-1$ elements, there are $m$ of these sets and $|H_i|=n_{k,m}=(k-1)m$, it follows that $V_1, \dots, V_m$ form a partition of $H_i$ into cliques of size $k-1$. 

        Every color may appear at most once in each clique and there are only $k-1$ colors in total. Therefore, in the coloring $\varphi$ every color appears exactly once in each $V_j$ and $m = \frac{n_{k,m}}{k-1}$ times in $H_1$. 

        \medskip

        \paragraph*{\textbf{Case~2.} $k$ is even.}

        We will show that every color appears at most $2m$ times in every set of $n_{k, m}$ consecutive unaffected vertices. As $n_{k,m} = 2m(k - 1)$ and there are $k-1$ colors, this then implies that every color must appear exactly $\frac{n_{k,m}}{k-1}=2m$ times, as desired.

        Recall that in the odd case, we partitioned $H_i$ into cliques and used that every color may appear at most once in each clique. The even case is more difficult, but follows a similar idea. In the beginning, we will restrict the vertices of a color to $10m$ options in $H_i$. To show that there are at most $2m$ vertices of the same color, we then partition these $10m$ vertices into $m$ subgraphs on $10$ vertices, each of which will have a spanning subgraph isomorphic to the graph $H$ depicted in Figure~\ref{fig:H}.

        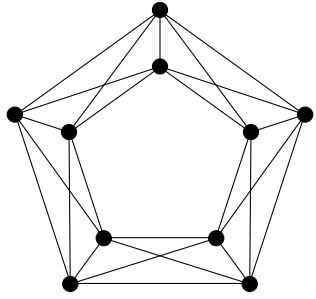
\begin{figure}[ht]
            \centering
        \begin{tikzpicture}
            \node[name=l_0, regular polygon, regular polygon sides=5, minimum size = 2.5cm, draw] at (2,0) {};
            \node[name=l_1, regular polygon, regular polygon sides=5, minimum size = 4cm, draw] at (2,0) {};
            \foreach \x in {1,...,5}
            {
                \draw[fill=black] (l_0.corner \x) circle (0.1cm);
                \draw[fill=black] (l_1.corner \x) circle (0.1cm);
                \draw (l_0.corner \x) -- (l_1.corner \x);
            }
            \draw (l_0.corner 1) -- (l_1.corner 2);
            \draw (l_0.corner 1) -- (l_1.corner 5);
            \draw (l_0.corner 2) -- (l_1.corner 1);
            \draw (l_0.corner 2) -- (l_1.corner 3);
            \draw (l_0.corner 3) -- (l_1.corner 2);
            \draw (l_0.corner 3) -- (l_1.corner 4);
            \draw (l_0.corner 4) -- (l_1.corner 3);
            \draw (l_0.corner 4) -- (l_1.corner 5);
            \draw (l_0.corner 5) -- (l_1.corner 1);
            \draw (l_0.corner 5) -- (l_1.corner 4);
        \end{tikzpicture}
        \caption{The graph $H$.}\label{fig:H}
        \end{figure}

        It is easy to see that $H$ has independence number two, and thus every color may appear at most twice in each of these $10$-vertex subgraphs. This will then imply that every color appears at most $2m$ times in $H_i$.

        Consider $H_i \subseteq S$ and a color $c$. Our goal is to show that $c$ appears at most $2m$ times in $H_i$. If it does not appear at all, this is trivially satisfied, so suppose that it appears at least once. Let $j$ be minimal such that $v_j \in H_i$ and $\varphi(v_j) = c$. Let $A$ be the set of indices of the vertices of color $c$ in $H_i$. We will now use the definition of the edge set to restrict $A$ to $10m$ options. We note that the set we define here is larger than the one that immediately follows from the forbidden distances in $D$, but it will be more convenient to work with.
        \begin{align*}
            A &\subseteq j + ([0, n_{k,m} - 1] \setminus D)\\
            &\subsetneq j + ([0, 2m - 1] \cup [(k-4)m, (k+2)m-1] \cup [(2k - 4)m, (2k - 2)m - 1])\\ &=: U.
        \end{align*}

        By symmetry, we may assume without loss of generality that $U \subseteq S$: If not, take $j$ to be maximal with $v_j \in H_i$ and $\phi(v_j)=c$ and then define $U$ by subtracting the above intervals from $j$ rather than adding them. We note that for all $v_j \in S$ at least one of $v_{j + (2k - 2)m - 1}$ and $v_{j - (2k - 2)m + 1}$ is also in $S$.

        We define        
        \begin{equation}\notag
            U_0 := \{j, j + 1\} + \{0, (k - 4)m, (k - 2)m, km, (2k-4)m\}
        \end{equation}

        and for all $l \in [0,m - 1]$ let $B_l := 2l + U_0$ and
        $V_l := \{v_x | x \in B_l\}$.
    
        Our goal is to show that the $V_l$ partition the vertices with indices in $U$ and that $G'[V_l]=G[V_l]$ has a spanning subgraph isomorphic to $H$. To prove that $V_0,V_1,\ldots,V_{m-1}$ form a partition of $U$, we observe that for all $x \in U$ we have $x \in B_l$ if and only if $x - j \equiv 2l \pmod{2m}$ or $x - j \equiv 2l + 1 \pmod{2m}$. 
        
        Now let any $l\in [0,m-1]$ be given and let us find a spanning subgraph of $G[V_l]$ isomorphic to $H$. To do so, we label the vertices in $V_l$ as $\{u_0, u_0', u_1, u_1', u_2, u_2', u_3, u_3', u_4, u_4'\}$, increasingly along the linear ordering $(v_0,v_1,\ldots,v_{N-1})$. Then we claim that for $a < b$ such that $b - a \in \{1, 4\}$ the four vertices $u_a, u_a', u_b, u_b'$ form a clique. Indeed, the pairwise distances of such vertices are in $\{1\}\cup(\{-1, 0, 1\} + \{2m, (k - 4)m, (2k - 4)m\}) \subseteq D$. Therefore, and since the vertices in $U$ are all unaffected, the vertices in $V_l$ induce a graph which has $H$ as a subgraph.

        Using this fact it follows that the color $c$ can appear at most $2$ times on each set $V_l$. Since by what we have shown above $V_0\cup \cdots \cup V_{m-1}=\{v_j|j\in U\}$ contains all vertices of color $c$ in $H_i$, it follows that color $c$ appears in total at most $2m$ times on vertices in $H_i$ under $\phi$. But since there are $(k-1)$ colors $c\in [1,k-1]$ and since $|H_i|=n_{k,m}=2m(k-1)$, this also implies that every color appears \emph{exactly} $2m$ times in $H_i$ in the coloring $\phi$.

        Summarizing, we have now proved that on every set of $n_{k,m}$ consecutive unaffected vertices in $S$, every color appears exactly $2m = \frac{n_{k,m}}{k-1}$ times, as desired.        
    \end{claimproof}

    Recall that our goal is to show that $\varphi$ is periodic on $S'$, and we already know that it is periodic on $S$. Therefore, it is natural to introduce the expected color of a vertex: The color it would have if the coloring on $S$ was continued periodically with period $n_{k,m}$. More formally, let $i \in [0,\ldots,N-1]$. We define the \emph{expected color} $\varphi_e(v_i)$ to be the color of all $v_{i_0} \in S$ such that $i \equiv i_0 \pmod{n_{k,m}}$. This is always well-defined by Claim~\ref{local-periodicity-claim}. 
    
    To write expressions in the following more conveniently, we shall consider an extension of the vertex-set $\{v_0,\ldots,v_{N-1}\}$ of $G$ to an infinite set $\{v_i|i\in\mathbb{Z}\}$ and to extend the definition of $\phi_e(v_i)$ to this set periodically: For every $i\in \mathbb{Z}$ we set $\phi_e(v_i):=c$ where $c\in [1,k-1]$ is the color of all $v_{i_0} \in S$ such that $i \equiv i_0 \pmod{n_{k,m}}$ under $\phi$.

    \begin{remark}\label{local-periodicity-rem}
        It follows immediately from the definition and Claim~\ref{local-periodicity-claim} that the mapping defined by $i\rightarrow \phi_e(v_i)$ is periodic on $\mathbb{Z}$ with period $n_{k,m}$ and that every $n_{k,m}$ consecutive vertices of $G'$ contain exactly $\frac{n_{k,m}}{k-1}$ vertices of each expected color.
    \end{remark}

    Our next intermediate goal will be to show that the coloring $\varphi$ has a structure resembling that of $\varphi_J$ on $S$. As a consequence, that structure will also hold for the expected colors $\varphi_e$. 

    \begin{claim}[Local Structure]\label{structure-claim}
        If $k \geq 6$ then there exist $t_{\text{even}}, t_{\text{odd}} \in \mathbb{Z}$ such that $t_{\text{even}}$ is even, $t_{\text{odd}}$ is odd, and for all $i \in \mathbb{Z}$ we have 
        \begin{equation}\notag
        \varphi_e(v_i) \neq \varphi_e(v_{i - 2}) \iff i \in \{t_{\text{even}}, t_{\text{odd}}\} + 2m\mathbb{Z}.
        \end{equation} 
    \end{claim} 
    As one notices, the claim does not address the case $k=5$. This is fine, as the claim will not be necessary in the later steps of the proof for $k=5$.
    
    In the following, we give a proof of Claim~\ref{structure-claim} for most cases, namely when $k\notin \{6,8\}$. The cases $k\in \{6,8\}$ are more technical to handle and, as previously explained, we defer the proof of Claim~\ref{structure-claim} in these cases to Appendix~\ref{sec:68}. In the following, we will first prove the slightly different claim below, and later show that it implies Claim~\ref{structure-claim} for $k\notin \{6,8\}$. Recall that $F=\{0,2,\ldots,2m-2\}$.
    \begin{claim}\label{equiv-structure}
        Let $k\ge 7, k\neq 8$. For every color $c$ there exist $t_c, t_c'$ (not necessarily distinct) such that for every $i\in \mathbb{Z}$, we have $\varphi_e(v_i) = c$ if and only if $i \in T_c :=\{t_c, t_c'\} + F + n_{k,m}\Z$. Moreover, $t_c=t_c'$ if $k$ is odd and $t_c\neq t_c'$ if $k$ is even.
    \end{claim}
        
    \begin{claimproof}[Proof of Claim~\ref{equiv-structure}]
        The proof has two cases depending on the parity of $k$.

        \medskip
        
        \paragraph*{\textbf{Case~1.} $k$ is odd and $k \geq 7$.} By Claim~\ref{local-periodicity-claim} we know that color $c$ appears $m$ times in $\{v_1,\ldots,v_{n_{k,m}}\}$. Let $j\in [1,n_{k,m}]$ be minimal such that $\phi_e(v_j)=\phi(v_j)=c$. 
        Again by Claim~\ref{local-periodicity-claim} there are exactly $m$ vertices $v_l$ of color $c$ in $H_{j+2m}$, i.e.,  such that $l - j \in [2m, n_{k, m} + 2m-1]$. Now, consider any such vertex $v_l$. Using edges from $v_j$ of distances in $D_2$ we can conclude $l-j \in [n_{k, m} - 2m + 2, n_{k, m} + 2m-1]$, and thus $l\in B:=j+[n_{k, m} - 2m + 2, n_{k, m} + 2m-1]$, a range of $4m - 2$ integers. Summarizing, we have established that all $m$ vertices of color (equivalently, expected color) $c$ in $H_{j+2m+1}$ are contained in $\{v_l|l\in B\}\subseteq S$. 
        
        Let $t_c \in B$ be minimal such that $v_{t_c}$ has color $c$. If the remaining $m-1$ vertices of color $c$ with indices in $B$ are $v_{t_c+2}, v_{t_c+4}, \dots, v_{t_c+2m - 2}$, then by periodicity, for all $i\in \mathbb{Z}$ we have $\varphi_e(v_i) = c$ if and only if $i \in t_c + F + n_{k,m}\mathbb{Z}$. 
        
        Otherwise, there exists $d > 2m - 2$ such that $v_{t_c + d} \in B$ and $\varphi_e(v_{t_c + d})=\phi(v_{t_c + d}) = c$. By definition of $D_2$ we have $d > (k - 3)m$. As $k \geq 7$, we conclude $d > 4m$ and therefore $v_{t_c + d} \notin B$, contradicting the definition of $d$. 
        
        It follows that indeed for every color $c\in [1,k-1]$ there exists some $t_c\in \mathbb{Z}$ such that for every $i\in \mathbb{Z}$, we have that $v_i$ has expected color $c$ if and only if $i \in t_c + F + n_{k,m}\mathbb{Z}$. This concludes the proof by setting $t_c':=t_c$.

        \medskip

        \paragraph*{\textbf{Case~2.} $k$ is even and $k \geq 10$.} 
        As in the first case, let $j \in [1, n_{k,m}]$ be minimal such that and $\varphi_e(v_j)=\phi(v_j)=c$. By Claim~\ref{local-periodicity-claim} there are exactly $2m$ vertices of (expected) color $c$ in $H_{j + 2m}$. 
        Using the edges from $v_j$ with distances in $D$ we can conclude that each such vertex $v_l$ must satisfy
        \begin{equation}\notag
            l - j \in [(k-4)m + 3, (k+2)m - 2] \cup [(2k - 4)m + 2, 2km - 1],
        \end{equation}
        a union of two sets of fewer than $6m$ consecutive vertices. 

        As $6 \leq k - 4$, any pair of vertices $v_x, v_y$ of the same expected color with $d_N(x, y) < 6m$ satisfies, due to the distances in $D_1 \cup D_2$, that $d_N(x, y) \in F$. In particular, any set of $6m$ consecutive vertices can contain at most $m$ vertices of the same expected color. In particular, this is true for $C_1 := [(k-4)m + 3, (k+2)m - 2]$ and $C_2 := [(2k - 4)m + 2, 2km - 1]$.

        As by the above there are in total exactly $2m$ vertices of (expected) color $c$ in $C_1\cup C_2$, this implies that each of $C_1$ and $C_2$ contain exactly $m$ indices of vertices of expected color $c$. Let $t_c \in C_1$ and $t_c' \in C_2$ be minimal in their sets such that $\varphi_e(t_c) = c$ and $\varphi_e(t_c') = c$. By the above, for every vertex $v_l$ of (expected) color $c$ with $l\in C_1$ we then have $d_N(t_c,l)\in F$ and thus $l\in t_c+F$, and similarly for every vertex $v_l$ of (expected) color $c$ with $l\in C_2$ we have $d_N(t_{c}',l)\in F$, meaning $l\in t_c'$. 

        Altogether, this establishes that every vertex $v_l\in H_{j+2m}$ of (expected) color $c$ satisfies $l\in \{t_c,t_c'\}+F$.
        
        Since the set $\{t_c,t_c'\}+F$ has size $2m$ and there are exactly $2m$ vertices of color $c$ in $H_{j+2m}$, it follows that in fact, for every $v_i\in H_{j+2m}$ we have $\phi_e(v_i)=c$ \emph{if and only if} $i\in \{t_c,t_c'\}+F$. By periodicity of $\varphi_e$ we can generalize this to: For all $i \in \mathbb{Z}$ we have $\varphi_e(v_i) = c$ if and only if $i \in \{t_c, t_c'\} + F + n_{k,m}\mathbb{Z}$.

        This concludes the proof of Claim~\ref{equiv-structure} when $k$ is even.
    \end{claimproof} 

    Next, we will use Claim~\ref{equiv-structure} to prove Claim~\ref{structure-claim}.

    \begin{claimproof}[Proof of Claim~\ref{structure-claim}]        
        We always have $\varphi_e(v_j) \neq \varphi_e(v_{j + 2m})$ because $2m \in D_2$. In particular, this tells us that for every $j \in \mathbb{Z}$ there exists a $t\in j + 2 + F$ such that $\varphi_e(v_t) \neq \varphi_e(v_{t - 2})$. Applying this with $j=-2$ and $j=-1$ respectively, we find that that there exist numbers $t_{\text{even}}\in F$ and $t_{\text{odd}}\in F+1$ such that $\varphi_e(v_{t_{\text{even}}})\neq \varphi_e(v_{t_{\text{even}}-2})$ and $\varphi_e(v_{t_{\text{odd}}})\neq \varphi_e(v_{t_{\text{odd}}-2})$.
        
        Consider any distinct numbers $i, j \in \mathbb{Z}$ of the same parity such that $\varphi_e(v_j) \neq \varphi_e(v_{j - 2})$, and $\varphi_e(v_i) \neq \varphi_e(v_{i - 2})$. We now claim that for every such pair of numbers we must have $|i - j| \geq 2m$. To see this, w.l.o.g. assume $i < j$ and towards a contradiction, suppose that $j - i < 2m$. Set $c:=\varphi_e(v_i)$. Since $\varphi_e(v_{i - 2})\neq \varphi_e(v_i)=c$, by Claim~\ref{equiv-structure} we have $i\in T_c=\{t_c,t_c'\}+F+n_{k,m}\mathbb{Z}$ and $i-2\notin T_c=\{t_c,t_c'\}+F+n_{k,m}\mathbb{Z}$. Using the definition of $F$, this implies $i \in \{t_c, t_c'\} + n_{k,m}\mathbb{Z}$. Since we have $j-i, (j-2)-i\in \{0,2,\ldots,2m-2\}=F$ by assumption, this implies that $j, j-2 \in T_c$, so by Claim~\ref{equiv-structure} $j$ and $j - 2$ have the same expected color. This contradicts our assumption on $j$, and so we indeed have $|i - j| \geq 2m$, as claimed.

        Now let $j \in 2\mathbb{Z}$ be minimal such that $j > t_{\text{even}}$ and $\varphi_e(v_j) \neq \varphi_e(v_{j - 2})$. Our previous observations imply that $j - t_{\text{even}} \geq 2m$ and $\varphi_e(v_{t_{\text{even}} + 2m}) \neq \varphi_e(v_{t_{\text{even}}})$. Therefore, we can see that $j = t_{\text{even}} + 2m$. Repeating this argument inductively and using periodicity we obtain that for every $i \in 2\mathbb{Z}$ we have $\varphi_e(v_i) \neq \varphi_e(v_{i - 2})$ if and only if $i \equiv t_{\text{even}} \pmod{2m}$. 

        By the same argument, we get that for every $i \in \mathbb{Z}$ which is odd we have $\varphi_e(v_i) \neq \varphi_e(v_{i - 2})$ if and only if $i \equiv t_{\text{odd}} \pmod{2m}$. 

        In conclusion, it follows that for all $i \in \mathbb{Z}$ we have $\varphi_e(v_i) \neq \varphi_e(v_{i - 2})$ if and only if $i \in \{t_{\text{even}}, t_{\text{odd}}\} + 2m\mathbb{Z}$, concluding the proof of the claim. 
    \end{claimproof}

    Next, we will show that the unaffected vertices in $S$ already exclude all but two color options for every vertex in $S'$. We will use that while some edges from $G$ might be missing in $G'$, none of the edges with at least one endpoint in $S$ are missing. In particular, all edges between $S$ and $S'$ are present.

    \begin{claim}[Two Color Options]\label{two-colors-claim}
        Let $i \in \left[1, \frac{N-1}{2}\right]$. Then $v_i$ satisfies $$\varphi(v_i) \in \{\varphi_e(v_i), \varphi_e(v_{i - 1})\}.$$
    \end{claim}
    \begin{claimproof}
        If $i \leq 3n_{k,m}$ then $v_i \in S$ and $\varphi(v_i) = \varphi_e(v_i)$ by definition of $\varphi_e$. If $i > 3n_{k,m}$ we will prove that $v_i$ has neighbors in $S$ of every color except possibly $\varphi_e(v_i)$ and $\varphi_e(v_{i-1})$, implying that $\phi(v_i)\in \{\phi_e(v_{i-1}),\phi_e(v_i)\}$, as desired. 

        \medskip

        \paragraph*{\textbf{Case~1.} $k$ is odd.} 

        Let $i_0 \in [n_{k,m} + 1, 2n_{k, m}]$ be such that $i \equiv i_0 \pmod{n_{k,m}}$ and let
        $$J = \{0, 2m, 4m, \dots, (k - 3)m\}.$$ We recall that the vertices in $$V_J := \left\{\, v_{i_0 - d}, v_{i_0 - 1 - d} \mid d \in J \,\right\} \subseteq S$$ form a clique in $G$ (and thus in $G'$): The pairwise distances are $1, 2m - 1 \in D_1$ or in $[2m, (k - 3)m + 1] \subseteq D_2$. The clique $V_J$ has $k - 1$ unaffected vertices, so every color appears in it.
        All vertices $v_j \in V_J\setminus\{v_{i_0}, v_{i_0 - 1}\}$ satisfy 
        \begin{equation}\label{eq-remainder}
            i_0 - j \in [2m, (k - 3)m + 1].
        \end{equation}
        Using that $i \leq \frac{N-1}{2}$ and $i - i_0 \in n_{k,m}\{1,2\ldots,\frac{q}{2}-1\}$, (\ref{eq-remainder}) tells us that $d_N(j, i)=i-j=(i_0-j)+(i-i_0) \in D_2$.
        Therefore, $v_i$ has neighbors in $S$ of all colors except possibly $\varphi_e(v_{i_0})$ and $\varphi_e(v_{i_0 - 1})$, as desired. This concludes the proof in the first case. We note that in this case, we did not use Claim~\ref{structure-claim}, so it is indeed valid for all odd $k$, including $k=5$.

        \medskip

        \paragraph*{\textbf{Case~2.} $k$ is even.}

        Let $i_0 \in [n_{k, m} + 1, 2n_{k, m}]$ be such that $i_0 \equiv i  \pmod{n_{k, m}}$. Let $q' :=\frac{i - i_0}{n_{k, m}} \in \mathbb{Z}$ and note that $q' < \frac{q}{2}$ since $i\le \frac{N-1}{2}$. Using edges of distances \begin{equation}\notag
            q'n_{k, m} + [2m, (k - 4)m + 2] \subseteq D_2
        \end{equation} 
        and 
        \begin{equation}\notag
            (q'-1)n_{k, m} + [(k + 2)m - 1, (2k - 4)m + 1] \subseteq D_3
        \end{equation}
        we can see that $v_{i}$ is adjacent in $G$ to all vertices $v_j$ satisfying:
        \begin{equation}\label{forbidden-difs-eqn}
            j - i_0 \in [- (k-4)m - 2, - 2m] \cup [2m - 1, (k-4)m + 1].
        \end{equation}
Since these vertices are all contained in $S$, it follows that $v_i$ is also in $G'$ adjacent to all these vertices.
        
        Let us now consider any color $c\in [1,k-1]$ such that none of the $2m$ vertices of color $c$ in $H_{i_0 - (k - 4)m - 2}$ are adjacent to $v_i$. Then by (\ref{forbidden-difs-eqn}) all of those $2m$ vertices must be of the form $v_j$ where 
        \begin{equation}\notag
            j - i_0 \in [- 2m + 1, 2m - 2] \cup [(k-4)m + 2, (k + 2)m - 3].
        \end{equation}
        We recall that $F=\{0, 2, 4, \dots, 2m -2\}$. By Claim~\ref{equiv-structure} there are $t_1, t_2 \in [i_0 - (k - 4)m - 2, i_0 + (k + 2)m - 3]$ such that $t_1 < t_2$ and for every $l\in \mathbb{Z}$ we have $\varphi_e(v_l) = c$ if and only if $l \in \{t_1, t_2\} + F + n_{k,m}\mathbb{Z}$. 
        As no vertex with index in $t_1 + F$ or $t_2 + F$ should be a neighbor of $v_i$, we get that 
        \begin{equation}\notag
            t_1, t_2 \in i_0 + ([-2m + 1, 0] \cup [(k-4)m + 2, km - 1]). 
        \end{equation}
        Let $T_1 = i_0 + [-2m + 1, 0]$ and $T_2 = i_0 + [(k-4)m + 2, km - 1]$. We note that $|T_1| = 2m$ and $|T_2| = 4m - 2$. 

        Next, we claim that the difference $t_2 - t_1$ must be at least $4m$. 
        
        To see this, consider first the case that $t_2 - t_1$ is even. Note that we have $\phi_e(v_l)=\phi_e(v_{l-2})$ for every $l\in \mathbb{Z}$ such that $l\in \{t_1,t_2\}+\{2,4\ldots,2m-2\}+n_{k,m}\mathbb{Z}$. Using this and Claim~\ref{structure-claim}, we may conclude $t_1 \equiv t_2 \pmod{2m}$. Note further that we have $t_2-t_1\neq 2m$ since $2m\in D_2$ and $v_{t_1}, v_{t_2}\in S$ satisfy $\phi(v_{t_1})=\phi_e(v_{t_1})=\phi_e(v_{t_2})=\phi(v_{t_2})=c$. Finally, this implies $t_2- t_1 \geq 4m$, as desired. 
        
        Next let us look at the case when $t_2 - t_1$ is odd. We consider the element $t_1 + 2m - 2 \in t_1 + F$. Every $t \in [t_1, t_1 + 4m -1]$ where $t - t_1$ is odd satisfies 
        \begin{equation}\notag
            |t - (t_1 + 2m - 2)| \in  \{1, 3, \dots, 2m + 1\} \subseteq D_1 \cup D_2.
        \end{equation}   
        Thus, $t_2\notin [t_1,t_1+4m-1]$ and so $t_2 - t_1 \geq 4m$, as desired.

        In both cases we conclude that $t_1\in T_1$ and $t_2\in T_2$, as neither of the intervals $T_1$, $T_2$ is long enough to host two elements with distance at least $4m$.

        Let $s_{\text{even}}, s_{\text{odd}} \in T_1$ be the unique integers satisfying $s_{\text{even}} \equiv t_{\text{even}} \pmod{2m}$ and $s_{\text{odd}} \equiv t_{\text{odd}} \pmod{2m}$. They are unique because $T_1$ consists of $2m$ consecutive integers. From Claim~\ref{structure-claim} we get 
        \begin{equation}\notag
            t_1 \in \{s_{\text{even}}, s_{\text{odd}}\}.
        \end{equation}
        Recall that $\phi_e(t_1)=c$, and so it follows that $c\in \{\phi_e(s_{\text{even}}),\phi_e(s_{\text{odd}})\}$. 
        
        Summarizing, up until this point, we have shown that every color $c\in [1,k-1]$ that in the coloring $\phi$ does not appear on a neighbor of $v_i$ in $G'$, must satisfy $c\in \{\phi_e(s_{\text{even}}),\phi_e(s_{\text{odd}})\}$. This implies that $\phi(v_i)\in \{\phi_e(s_{\text{even}}),\phi_e(s_{\text{odd}})\}$.

        To conclude the proof, we will now show that $\{\varphi_e(s_{\text{even}}), \varphi_e(s_{\text{odd}})\} = \{\varphi_e(v_{i}), \varphi_e(v_{i - 1})\}$. 

        We start by noting that, as $s_{\text{even}}, s_{\text{odd}} \in T_1 = [i_0 - 2m + 2, i_0]$ and they have different parities, one of them will be in 
        \begin{equation}\notag
            \{i_0 - 2m + 1, i_0 - 2m +3, \dots, i_0 - 1\} = (i_0 - 1) - F,
        \end{equation}
        while the other will be in 
        \begin{equation}\notag
            \{i_0 - 2m + 2, i_0 - 2m + 4, \dots, i_0\} = i_0 - F.
        \end{equation}

        Further, by Claim~\ref{structure-claim} we have \begin{equation}\notag
\phi_e(v_{s_{\text{even}}})=\phi_e(v_{s_{\text{even}}+2})=\cdots=\phi_e(v_{s_{\text{even}}+2m-2})\end{equation} and \begin{equation}\notag
\phi_e(v_{s_{\text{odd}}})=\phi_e(v_{s_{\text{odd}}+2})=\cdots=\phi_e(v_{s_{\text{odd}}+2m-2}).\end{equation}
        
        Thus, we can conclude that $\{\varphi_e(s_{\text{even}}), \varphi_e(s_{\text{odd}})\} = \{\varphi_e(v_{i_0}), \varphi_e(v_{i_0 - 1})\}$. Since the assignment $x\rightarrow \phi_e(v_x)$ is periodic with period $n_{k,m}$ by definition and $i\equiv i_0 \pmod{n_{k,m}}$, we in particular have $\{\varphi_e(s_{\text{even}}), \varphi_e(s_{\text{odd}})\}=\{\varphi_e(v_{i}), \varphi_e(v_{i - 1})\}$, as desired. This concludes the proof of the claim also in the second case when $k$ is even.
    \end{claimproof}

    The next claim uses the distances in $D_1$ to show that if enough vertices in $G'$ of the same color under $\phi$ are close to each other, then their indices must have the same parity. To simplify the writing, in the following we will say that two vertices $v_i$ and $v_j$ have the same parity if $i$ and $j$ have the same parity. We will also refer to a vertex $v_i$ as \emph{odd} if $i$ is odd and as \emph{even} if $i$ is even.

    \begin{claim}\label{2m-claim}
        Let $\overline{S}$ be a subset of $2m$ consecutive vertices in $G'$ such that for all $v_i \in \overline{S}$ we have $i \leq \frac{N-1}{2}$ and let $c\in [1,k-1]$ be any color. If more than $2r$ vertices of $\overline{S}$ have color $c$ under $\phi$, then all vertices of color $c$ in $\overline{S}$ have the same parity.
    \end{claim}
    \begin{claimproof}
        Assume that there are more than $2r$ vertices of color $c$ in $\overline{S}$. Without loss of generality, at least $r + 1$ of the $c$-colored vertices in $\overline{S}$ are even (if the majority of vertices is odd the argument proceeds analogously). All distances within $\overline{S}$ are smaller than $2m$ and the distances between vertices of different parities are odd. By definition of $D_1$ the odd vertices in $\overline{S}$ are adjacent to all even vertices in $\overline{S}$ in $G$. In particular, every odd vertex in $\overline{S}$ must be connected to all the (at least $r+1$) even vertices of color $c$ in $\overline{S}$ in $G$. This means that every odd vertex in $\overline{S}$ must have at least one neighbor of color $c$ in $G'$ as well and cannot have color $c$ itself. So all vertices in $\overline{S}$ of color $c$ are even and thus have the same parity.
    \end{claimproof}
        
    With all the previously established claims at hand, we are now finally ready to complete the proof Lemma~\ref{partial-periodicity-lemma}. 
    Recall that Lemma~\ref{partial-periodicity-lemma} states that $\varphi$ is periodic on all of $S'$. As $\varphi$ is periodic on $S$ and thus $\varphi_e$ is periodic on $S'$, it suffices to show that for all $v_j \in S'$ we have $\varphi(v_j) = \varphi_e(v_j)$.
    We will consider the vertex of smallest index which is not colored in its expected color. Applying the different claims to the vertices close to it, we will obtain a contradiction. 

    Towards a contradiction, assume there is a vertex in $S'=\left\{v_1,\ldots,v_{\frac{N-1}{4}}\right\}$ which does not have its expected color. Let $i\in \left[0,\frac{N-1}{4}-1\right]$ be minimal such that $v_{i+1} \in S'$ and $\varphi(v_{i+1}) \neq \varphi_e(v_{i+1})$. Since $\phi$ and $\phi_e$ agree on $S'$, this index $i$ is well-defined and we have $i\ge 3n_{k,m}$. Applying Claim~\ref{two-colors-claim} to $v_{i+1}$ and using the minimality of $i$ we conclude that $\varphi(v_{i+1}) = \varphi_e(v_i) = \varphi(v_i) =: c$.

    The main idea is to find approximately $m$ or $2m$ consecutive vertices which can only have color $c$ or one other color. We can then apply Claim~\ref{2m-claim} to $c$ and then the other color to find a contradiction.
    
    \medskip

    \paragraph*{\textbf{Case~1.} $k \geq 6$.}

    Without loss of generality, assume that $i$ is even. Let $t_0$ be the unique integer satisfying $t_0 \equiv t_{\text{even}} \pmod{2m}$ and $t_0 \in [i - 2m + 1, i]$.  Intuitively, $t_0$ is the "starting point" of color $c$ before (or at) $i$. Similarly, let $t_1$ be the smallest integer satisfying $t_1 \equiv t_{\text{odd}} \pmod{2m}$ and $t_1 \geq t_0$. Intuitively, $t_1$ is the first time after $t_0$ where the color on odd vertices changes. Let $c_1 := \varphi_e(v_{t_1 - 2})$ and $c_2 := \varphi_e(v_{t_1})$. We have $c_1\neq c_2$ by Claim~\ref{structure-claim}.

    Consider the $2m$ consecutive vertices in $\overline{S} := \{v_{t_0}, v_{t_0 + 1}, \dots, v_{t_0+2m -1}\}$. Then Claim~\ref{structure-claim} tells us that for all $v_j \in \overline{S}$ we have:
    \begin{equation}\notag
        \varphi_e(v_j) = \begin{cases}
            c &\text{if } j \text{ is even},\\
            c_1 &\text{if } j \text{ is odd and } j < t_1, \\
            c_2 &\text{if } j \text{ is odd and } j \geq t_1.  
        \end{cases}
    \end{equation}

    Then the vertices $v_j$ with $j \in [t_0 + 1, t_1 - 1] =: T_1$ satisfy, by Claim~\ref{two-colors-claim}:
    \begin{equation}\notag
        \varphi(v_j) \in \{\varphi_e(v_{j-1}), \varphi_e(v_j)\} = \{c, c_1\}.
    \end{equation}
    Similarly, for $j \in [t_1, t_2 + 2m - 1] =: T_2$ we have, using that $t_1 - 1$ is even:
    \begin{equation}\notag
        \varphi(v_j) \in \{c, c_2\}.
    \end{equation}
    We note that $|T_1| + |T_2| = 2m-1$. Let $T \in \{T_1, T_2\}$ be the larger of these, so $|T| \geq m$.
    Let $T'\subseteq T \subseteq \overline{S}$ be any subset of $m$ consecutive vertices of $T$. Then there exists some $c'\in \{c_1,c_2\}$ such that all vertices in $T$ (and thus $T'$) have color $c$ or $c'$ under $\phi$. We know that $\varphi(v_i) = \varphi(v_{i+1}) = c$, so not all vertices of color $c$ in $\overline{S}$ have the same parity. Thus, by Claim~\ref{2m-claim} at most $2r$ vertices in $\overline{S}$ have color $c$. Therefore, $T'$ forms a set of $m$ consecutive vertices in $\overline{S}$, at least $m - 2r$ of which have color $c'$. As $m - 2r > \frac{m}{2} + 1$ by our assumptions on $m$ in the lemma, there are at least two consecutive vertices of color $c'$ in $T'$. But we also have $m - 2r > 2r$, so by Claim~\ref{2m-claim} all vertices of color $c'$ in $T'$ should have the same parity, a contradiction.

    \medskip

    \paragraph*{\textbf{Case~2.} $k = 5$.}

    In this case, we do not have as much useful information about the structure of the coloring. We can still prove the lemma using the following observations: For $k=5$, we have $n_{k,m} = 4m$. By Remark~\ref{local-periodicity-rem} every set of $4m$ consecutive vertices in the ordering $(v_0,\ldots,v_{N-1})$ contains exactly $m$ vertices of each expected color. Next, we claim that two out of the four colors in $\{1,2,3,4\}$ appear only on even vertices as expected colors ("even colors") and the two remaining colors appear only on odd vertices ("odd colors") as expected colors: To see this, consider any two vertices $v_x, v_y$ of distinct parity with $x,y \in [1,4m]$. From Claim~\ref{local-periodicity-claim} we know that $\varphi_e(v_{x + 4m}) = \varphi_e(v_{x})$. As one of $x + 4m - y$ and $y - x$ must be smaller than $2m$ and both are odd, one of $d_N(x, y), d_N(x+4m, y)$ is an element of $D_1$. Since all of $v_x, v_y$ and $v_{x+4m}$ are contained in $S=\{v_1,\ldots,v_{12m}\}$, it follows that $\varphi_e(v_x)=\phi(v_x) \neq \phi(v_y)=\varphi_e(v_y)$. 

    Since $\phi_e$ is periodic with period $4m$, the above implies that also $\phi_e(v_x)\neq \phi_e(v_y)$ for any $x,y\in \mathbb{Z}$ with distinct parity. Thus, for every $c\in \{1,2,3,4\}$ all vertices of expected color $c$ have the same parity. It remains to be argued that exactly two colors appear on odd and exactly two appear on even vertices. Suppose not. Then there exists some $c\in \{1,2,3,4\}$ and some $b\in \{0,1\}$ such that all vertices $v_i$ with $i\equiv b\pmod{2}$ have expected color~$c$. This is clearly a contradiction, as for instance the vertices $v_{b+2}$ and $v_{b+2m+2}$ are both contained in $S$ and adjacent to each other in $G$ and thus in $G'$, so they cannot have the same color under $\phi$, and thus they also have distinct expected colors. 

    Moving on, let us consider two sets $S_L$ and $S_R$ of $2m$ consecutive vertices ending or starting in $v_{i-1}, v_i, v_{i+1}$ respectively:
    \begin{equation}\notag
        S_L := \{v_{j} \mid j \in [i - 2m + 2, i + 1]\}
    \end{equation}
    and
    \begin{equation}\notag
        S_R := \{v_{j} \mid j \in [i - 1, i + 2m - 2]\}.
    \end{equation}

    Note that $S_R$ might not be fully contained in $S'$, but it will be sufficient in the following arguments that all indices of the vertices in $S_L \cup S_R$ are smaller than $\frac{N}{2}$.

    Our goal is, as in the previous case, to find many vertices of the same color close to each other, contradicting Claim~\ref{2m-claim}. In this case, we will use more properties of the vertices near $v_i$. In particular, we start by getting more out of the minimality of $i$:

    We note that the minimality of $i$ implies that $v_{i + 1}$ is the only vertex in $S_L$ not colored in its expected color. By Claim~\ref{2m-claim} and since $\phi(v_i)=\phi(v_{i+1})=c$, there are at most $2r$ vertices of color $c$ in $S_L$, and so at most $2r - 1$ of expected color~$c$. The set $S_L \cup S_R$ consists of $4m-3$ consecutive vertices, and so by Remark~\ref{local-periodicity-rem} at least $m-3$ of them have expected color~$c$. Therefore, the set $S_R \setminus S_L$ contains at least $m - 3 - (2r - 1)$ vertices of expected color~$c$. Since $i \in S_L \cap S_R$ we can conclude that $S_R$ contains a set $S_c$ of at least $m - 2r - 1$ vertices of expected color~$c$. 
    
    Using Claim~\ref{2m-claim} and $\varphi(v_i) = \varphi(v_{i + 1}) = c$  we get that $S_R$ contains at most $2r$ vertices of color~$c$. As $\phi(v_{i+1})=c$ but $v_{i+1} \notin S_c$, we have that $S_c$ contains at most $2r - 1$ vertices of color~$c$. Without loss of generality let us assume in the following that $i$ is even (the argument in the case when the parity of $i$ is odd is fully analogous, switching odd and even in the relevant places of the following argument). In particular, since $\phi_e(v_{i})=c$, it follows by our observation at the beginning of the proof about parities of vertices of a given expected color, that \emph{all} vertices of expected color $c$ are even, so $c$ is an even color.
    
    Let $c_1 := \varphi(v_{i - 1})$ and note that $i - 1$ is odd while all elements of $S_c$ are even. Therefore, at most $2r - 1$ vertices in $S_c $ can have color~$c_1$, for otherwise $\{v_{i-1}\}\cup S_c\subseteq S_R$ would contain more than $2r$ vertices of color $c_1$, not all of the same parity, contradicting Claim~\ref{2m-claim}. 
    
    Note further that since $c_1=\phi(v_{i-1})=\phi_e(v_{i-1})$, we have that $c_1$ is an odd color. Let $c_2 \in \{1,2,3,4\}$ be the other odd color besides $c_1$. As the vertices in $S_c$ are even, by Claim~\ref{two-colors-claim} the ones that do not have color~$c$ must have an odd color. Combining the previous observations, we can conclude that at least $m - 2r - 1 - (2r - 1) - (2r - 1) = m - 6r + 1$ vertices in $S_c$ have color~$c_2$ (here we used that $S_c$ has size at least $m-2r-1$, and contains at most $2r-1$ vertices of color $c$ and at most $2r-1$ vertices of color $c_1$). For every such $v_j \in S_c$, by Claim~\ref{two-colors-claim} we must have $\varphi_e(v_{j - 1}) = c_2$. We conclude that in addition to the at least $m - 2r - 1$ vertices of expected color~$c$ in $S_R$, there are at least $m - 6r + 1$ vertices of expected color~$c_2$ in $S_R$ since $v_{i - 1} \notin S_c$.

    Let $S_1:=\{v_j\in S_R|c \in \{\varphi_e(v_{j - 1}), \varphi_e(v_j)\}\}$ and $S_2:=\{v_j\in S_R|c_2 \in \{\varphi_e(v_{j - 1}), \varphi_e(v_j)\}\}$. Then we have $|S_1|\ge 2(m - 2r - 1) - 1$, where the final $-1$ comes from the fact that $v_{i + 2m - 2}$ is even and could have expected color~$c$, while $v_{i + 2m - 1}$ is not considered anymore. Similarly, we obtain  $|S_2|\ge 2(m - 6r + 1)$.

    Let now $S^\ast \subseteq S_R$ be the set of all $v_j\in S_R$ such that $\{\varphi_e(v_{j - 1}), \varphi_e(v_j)\} = \{c, c_2\}$. Clearly, $S^\ast=S_1\cap S_2$ and thus we have \begin{equation}\notag
    |S^\ast|=|S_1|+|S_2|-|S_1\cup S_2|\ge |S_1|+|S_2|-|S_R|\ge  (2(m-2r-1)-1)+2(m-6r+1)-2m = 2m - 16r - 1.\end{equation}

    As $S^\ast \subseteq S_R$ and since at most $2r$ vertices in $S_R$ have color~$c$, at most $2r$ vertices of $S^\ast$ may have color~$c$. By Claim~\ref{two-colors-claim}, all the remaining at least $2m - 18r - 1$ vertices of $S_R$ must have color~$c_2$. As $2m - 18r - 1 > \frac{2m}{2} + 1$ by our assumption on $m$ in the lemma, $S_R$ contains both even and odd vertices of color~$c_2$. Additionally $2m - 18r - 1 > 2r$, contradicting Claim~\ref{2m-claim}. 

    \medskip

    Having obtained a contradiction in both Case~1 and~2, it follows that our assumption about the existence of a vertex $v\in S'$ such that $\varphi(v) \neq \varphi_e(v)$ was wrong. Thus indeed, we have $\phi(v)=\phi_e(v)$ for all $v\in S'$. Using Remark~\ref{local-periodicity-rem}, it follows that $\phi$ is indeed periodic on $S'$ with period $n_{k,m}$. 
    This completes the proof of Lemma~\ref{partial-periodicity-lemma}.
\end{proof}

\section{Conclusive remarks}\label{sec:conc}
The only case of Dirac's conjecture and Erd\H{o}s's problems that remains open is $k=4$. In particular, the intriguing question whether there exists a $(4,1)$-graph remains open. However, the circulant graphs $G_{k,m,q}$ considered here cannot settle this remaining case. This is due to the fact that for $k=4$ and $m\ge 2$ we have $D_2=D_3=\emptyset$, so the only distances we allow for edges are those in $D_1=\{1,\ldots,2m-1\}$. It is then not hard to check that the so-defined circulant graph has chromatic number $3$. Furthermore, in the $k=4$ case there is a fundamental problem with using Jensen's color pattern illustrated in Table~\ref{table-even-coloring} to color the vertex-deleted subgraphs $G_{k,m,q}-v$: This color pattern forbids all distances greater than $1$ for the edges, and taking only distance $1$ results in an odd cycle, which is $3$-colorable. This of course does not rule out that some other distance sets could yield circulant $(4,1)$-graphs, but we were not able to find them so far, and we lean towards thinking that they do not exist.

We conclude with an observation about necessary properties of $(4,r)$-graphs, which may help guide the search for such graphs when $r$ is small.
\begin{proposition}\label{4-r-properties}
Let $r\ge 1$ be an integer and let $G$ be any $(4,r)$-graph. Then $G$ has edge-connectivity (and thus minimum degree) at least $3r+3$. Furthermore, $G$ has maximum degree at most $|V(G)|-(2r+3)$. In particular, $G$ has at least $5r+6$ vertices.
\end{proposition}
\begin{proof}
Let us first show the lower bound on the edge-connectivity. Consider any partition $(X,Y)$ of $V(G)$ into disjoint non-empty sets. Let $x$ denote the number of edges in $G$ going between $X$ and $Y$. Our goal is to show that $x\ge 3r+3$. Note that since $G$ is $4$-vertex-critical, there exist proper $3$-colorings $c_X:X\rightarrow \{1,2,3\}$, $c_Y:Y\rightarrow \{1,2,3\}$ of $G[X]$ and $G[Y]$. Let $\pi \in S_3$ be a permutation chosen uniformly at random. Then $\pi\circ c_Y$ is also a proper coloring of $G[Y]$. Furthermore, for each of the $x$ edges in $G$ going between $X$ and $Y$ the probability that their endpoint in $X$ has the same color under $c_X$ as their endpoint in $Y$ under $\pi\circ c_Y$ equals $\frac{1}{3}$. Thus, the expected number of such edges is $\frac{x}{3}$. In particular, this means we can choose $\pi\in S_3$ such that there are at most $\frac{x}{3}$ monochromatic edges in the $3$-coloring of $G$ defined by the common extension of $c_X$ and $\pi\circ c_Y$. On the other hand, since $G$ is a $(4,r)$-graph there must be at least $r+1$ such edges in every $3$-coloring of $G$. It follows that $\frac{x}{3}\ge r+1$ and so $x\ge 3r+3$, as desired.

Next, let us prove the statement on the maximum degree. Towards a contradiction, suppose there is some $v\in V(G)$ with $|N_G(v)|\ge |V(G)|-(2r+2)$. Let $W:=V(G)\setminus (\{v\}\cup N_G(v))$, such that $|W|\le 2r+1$. We claim that for every vertex $w\in W$ the set $N_G(w)\cap W$ has an independent subset of size at least $r+1$. To see this, consider any proper $3$-coloring $c:V(G)\setminus \{w\}\rightarrow \{1,2,3\}$ of $G-w$. W.l.o.g., we have $c(v)=1$ and thus $c(N_G(v))\subseteq \{2,3\}$. Since $G$ is a $(4,r)$-graph at least $r+1$ neighbors of $w$ must have color $1$ in $c$. All of these must be contained in $W$ as well and clearly form an independent set. This shows the above assertion. Next consider a path in $G[W]$ of maximum length. Let $w_1,\ldots,w_t$ be its sequence of vertices. Then all neighbors of $w_t$ in $W$ are contained in $\{w_1,\ldots,w_{t-1}\}$. Thus, by the above $\{w_1,\ldots,w_{t-1}\}$ has an independent subset of size $r+1$. On the other hand, since $w_1,\ldots,w_{t-1}$ form a path they cannot contain an independent set larger than $\left\lceil \frac{t-1}{2}\right\rceil \le \left\lceil \frac{|W|-1}{2} \right\rceil\le r$, a contradiction. This concludes the proof.
\end{proof}
In particular, the sparsest option for a $(4,1)$-graph left by Proposition~\ref{4-r-properties} is that of a $6$-regular graph, and from a computational viewpoint it seems natural to restrict the search to such graphs first.
\begin{problem}
    Does there exist a $6$-regular $(4,1)$-graph?
\end{problem}
\medskip

\paragraph*{\textbf{Acknowledgments.} We would like to gratefully acknowledge Yuval Wigderson for a helpful discussion and pointing us to the reference~\cite{conlonfox}.}

\appendix

\section{Vertex-criticality: Proof of Lemma~\ref{colorability-lemma}}\label{sec:vertcrit}

In this section we supply the proof of Lemma~\ref{colorability-lemma} that was left out in Section~\ref{sec:main} for the sake of readability. To prove the lemma, we have to show that there exists a proper $(k-1)$-coloring of $G_{k,m,q}-v$ for every vertex $v$. To do so, we will use the same coloring as Jensen~\cite{jensen}. However, note that our graphs have many additional edges compared to Jensen's graphs, so Lemma~\ref{colorability-lemma} does not immediately follow from the fact that the graphs in \cite{jensen} have this property.

\begin{proof}[Proof of Lemma~\ref{colorability-lemma}]
In the rest of the proof, let us abbreviate $G:=G_{k,m,q}$ and $D:=D_{k,m,q}$. Since $G_{k,m,q}$ is a circulant graph and thus vertex-transitive, it is sufficient to show that $G - v_0$ is $(k-1)$-colorable. We will show that the $n_{k,m}$-periodic coloring $\varphi_J$ used by Jensen~\cite{jensen} for his somewhat sparser graphs still forms a proper coloring for our graphs.

    Before defining $\phi_J$ formally, we set the following notation (the same that is used in the proof of Lemma~\ref{partial-periodicity-lemma}):
    \begin{equation}\notag
        H_i := \{v_j \mid j \in [i, i + n_{k,m} - 1]\}
    \end{equation} 
    for $i \in [1, (q-1)n_{k,m} + 1]$ and
    \begin{equation}\notag
        F := \{0, 2, \dots, 2m - 2\}. 
    \end{equation}

    Intuitively, $\varphi_J$ is a $n_{k,m}$-periodic coloring which is given on $H_1$ as follows: Each color has one (if $k$ is odd) or two (if $k$ is even) "starting points" in $H_1$. Then, all vertices $v_i$ such that $i$ is the sum of a starting point and an element of $F$ are assigned that color.

    The proof will be split into two cases depending on the parity of $k$, since the colorings are constructed slightly differently.

    \medskip
    
    \paragraph*{\textbf{Case~1.} $k$ is odd.}

    We define for colors $c \in [1, k - 1]$ the following "starting point":
    \begin{equation} \notag
        t_c := 
        \begin{cases}
        (c - 1)m + 1 & \text{ if } c \text{ is odd} \\
        (c - 2)m + 2 & \text{ if } c \text{ is even}.
        \end{cases}
    \end{equation} 
    For every $i \in [1, N-1]$ we set $\varphi_J(v_i) := c$ if and only if there exists $f \in F$ such that $i \equiv t_c + f \pmod{n_{k, m}}$. We will now show that this coloring is well defined. 

    We first check that every vertex is assigned at most one color this way: Suppose towards a contradiction that there exists $i \in [1, N-1]$ such that $i \equiv t_a + f_a \pmod{n_{k, m}}$ and $i \equiv t_b + f_b \pmod{n_{k, m}}$ where $f_a, f_b \in F$ and $a, b$ are distinct colors. Without loss of generality, we may assume $1 \leq i \leq n_{k, m}$. We note that $t_c + f \leq ((k-3)m + 2) + (2m - 2) = n_{k,m}$ for all colors $c$ and $f \in F$. Since $1 \leq t_a + f_a \leq n_{k, m}$ and $1 \leq t_b + f_b \leq n_{k,m}$ we have $t_a + f_a = i = t_b + f_b$. As all elements of $F$ are even, we know that $t_a$ and $t_b$ must have the same parity. From their definitions it then follows that $|t_a - t_b| \geq 2m$, which is larger than the difference of any elements of $F$, a contradiction since $|t_a-t_b| = |f_b - f_a|$.

    Next we check that every vertex is also assigned at least one color: As shown above we have $t_c + F \subseteq [1, n_{k, m}]$ for each color~$c$. There are $k-1$ colors and $|t_c+F|=|F|=m$ for every $c\in [1,k-1]$. As we know that no vertex receives more than one color, we can conclude that each of the $n_{k,m} = (k-1)m$ vertices in $H_1$ also has at least one color. By the periodic nature of the coloring this means every vertex, including the ones outside of $H_1$, has exactly one color.

    We will now show that $\varphi_J$ is a proper $(k-1)$-coloring of $G - v_0$: Towards a contradiction, assume that there are adjacent vertices $v_i, v_j$ in $G - v_0$ such that $\varphi_J(v_i) = c = \varphi_J(v_j)$ for some color~$c$. Write $i = a_i n_{k,m} + t_c + f_i$ and $j = a_j n_{k, m} + t_c + f_j$ where $f_i, f_j \in F$ and $a_i, a_j \in \mathbb{Z}$. We note that $0 \leq a_i, a_j < q$. Without loss of generality, let $i < j$. 
    
    Taking the difference between the above equations, we obtain:
    \begin{equation}\notag
        j - i = (a_j - a_i) n_{k, m} + f_j - f_i,
    \end{equation}
    and thus
    \begin{equation}\notag
        j - i \equiv f_j - f_i \pmod{n_{k, m}}.
    \end{equation}
    As $n_{k, m}$ and all elements of $F$ are even, it follows that $|i - j|$ is even. Therefore, $j - i \notin D_1$. Using that $f_j,f_i \in F$ we have: 
    \begin{equation}\notag
        f_j - f_i \in [-2m + 2, 2m - 2].
    \end{equation}
    In particular, the residue of $j - i$ modulo $n_{k, m}$ is not in $[2m, (k - 3)m + 1]$. In conclusion, $j - i \notin D$. Since we assumed that $v_i$ and $v_j$ are adjacent in $G$, we have $d_N(j,i)=\min\{j-i,N-(j-i)\}\in D$. Thus, the only possibility is that $j - i > \frac{N}{2}$ and $d_N(j,i)=N-(j-i)$. Since
    \begin{equation}\notag
        N - (j - i) = (q - a_j + a_i)n_{k, m} + (1 - f_j + f_i),
    \end{equation}
    we obtain
    \begin{equation}\notag
        d_N(j, i) \equiv 1 - f_j + f_i \pmod{n_{k, m}}.
    \end{equation}
    We note that $1 - f_j + f_i \in [-2m + 3, 2m - 1]$, which implies $d_N(j, i) \notin D_2$. 
    
    Recall that $0 \leq a_i, a_j < q$. Thus we obtain, using $k\ge 5$:
    \begin{equation}\notag
        d_N(j, i) \geq (q-(q-1) + 0)n_{k,m} + (-2m + 3) = (k - 3)m + 3 \geq 2m + 3.
    \end{equation}
    
    This implies that also $d_N(j,i) \notin D_1$. Since $D=D_1\cup D_2$, we obtain a contradiction to $d_N(j,i)\in D$. This shows that our initial assumption about the existence of the monochromatic edge $v_iv_j$ under $\phi_J$ was wrong. Thus, $\phi_J$ is indeed a proper $(k-1)$-coloring of $G-v_0$, concluding the proof in Case~1.

\medskip

\paragraph*{\textbf{Case~2.} $k$ is even.}

    We begin by defining the coloring $\varphi_J$ of $G - v_0$. Then we will show that it is a well-defined and proper coloring.

    For $c \in [1, k - 1]$ we define a set of two "starting points" as follows:
    \begin{equation}\label{def-starts-even}
    T_c := 
    \begin{cases}
        \left\{(c - 1)m + 1, (c + k - 3)m + 2\right\} & \text{if } c \text{ is odd} \\
        \left\{(c - 2)m + 2, (c + k - 2)m + 1\right\} & \text{if } c \text{ is even}.
    \end{cases}
    \end{equation}
 
    We now define the coloring $\varphi_J$ by setting: $\varphi_J(v_i) = c$ if and only if there exist $s \in T_c$ and $f \in F$ such that 
    \begin{equation}\notag
        i \equiv s + f \pmod{n_{k, m}}.
    \end{equation}

    We will show that this coloring is well-defined, i.e. the above rules assign exactly one color to each vertex. We do this by showing that $T = T'$, where
    \begin{equation}\notag
        T := \bigcup_{c \in [1, k-1]} T_c 
    \end{equation}
    and
    \begin{equation}\notag
        T' := (\{1, 2\} + 2m\mathbb{Z})\cap [1, n_{k,m}].
    \end{equation}
    We note that if $T = T'$ then every vertex $v_i \in H_1$ will satisfy $|(i - F) \cap T| = 1$, so it will have exactly one color assigned. Since the definition of the coloring is periodic with period $n_{k,m}$, this will then also imply that every vertex in $G-v_0$ is assigned exactly one color.
    
    One easily checks using the definition and that $k$ is even that $T\subseteq T'$. Since furthermore $T$ is a union of the $(k-1)$ $2$-element sets $T_1,\dots, T_{k-1}$ and since one easily checks that $T'$ has exactly $2(k-1)$ elements (recall $n_{k,m}=2(k-1)m$), it will be sufficient to show that $T_1,\dots, T_{k-1}$ are pairwise disjoint to establish that $T=T'$. Towards a contradiction, assume there are $c_1\neq c_2 \in [1, k-1]$ satisfying $T_{c_1} \cap T_{c_2} \neq \emptyset$. As the expressions defining the elements of $T_c$ are strictly increasing for vertices of the same parity, we can conclude that $c_1$ and $c_2$ have different parities. Let $t \in T_{c_1} \cap T_{c_2}$ and without loss of generality, let $c_1$ be odd. If $t$ is odd, then~(\ref{def-starts-even}) yields
    \begin{equation}\notag
        (c_1 - 1)m + 1 = (c_2 + k - 2)m + 1.
    \end{equation}
    Equivalently, $c_1 = c_2 + k - 1$. This is a contradiction, as $c_1, c_2 \in [1, k - 1]$. Similarly, if $t$ is even, we get $c_2 - 2 = c_1 + k - 3$, which is also a contradiction. All in all, we conclude that the coloring $\varphi_J$ is well-defined.

    We will now show that $\varphi_J$ is a proper coloring. 

    Equivalently, we may show that if $i, j \in [1, N-1]$ are such that the vertices $v_i$ and $v_j$ have the same color, then $d_N(i, j) \notin D$. So let $i < j \in [1, N-1]$ and $c \in [1, k-1]$ be such that $\varphi_J(v_i) = c = \varphi_J(v_j)$. By definition of $\varphi_J$ this implies that there exist $t_i, t_j \in T_c$ and $f_i, f_j \in F$ such that $i \equiv t_i + f_i \pmod{n_{k,m}}$ and $j \equiv t_j + f_j \pmod{n_{k,m}}$. Then 
    \begin{equation}\notag
    j - i \equiv t_j - t_i + f_j - f_i \pmod{n_{k,m}}.
    \end{equation}

    From (\ref{def-starts-even}) it follows that
    \begin{equation}\label{all-start-diffs}
        t_j - t_i \in \{-(k-2)m - 1, 0, (k-2)m + 1\} \cup \{-km + 1, 0, km - 1\}. 
    \end{equation}
    As $n_{k,m} = 2(k - 1)m = ((k - 2)m + 1) + (km - 1)$ we get that
    \begin{equation}\label{start-diffs}\notag
        t_j - t_i \in \{0, (k-2)m + 1, km - 1\} + n_{k,m}\mathbb{Z}. 
    \end{equation}

    This implies
    \begin{equation}\notag
        j - i \in A + n_{k,m}\mathbb{Z}
    \end{equation}
    where
    \begin{equation}\notag
     A := \{0, (k-2)m + 1, km - 1\} + F - F.
    \end{equation}

    We decompose $A = A_0 \cup A_1 \cup A_2$, where
    \begin{align*}
        A_0 &= \{0\} + F - F \\
        A_1 &= \{(k - 2)m + 1\} + F - F \\
        A_2 &= \{km - 1\} + F - F .
    \end{align*}
    Computing the elements explicitly, we obtain:
    \begin{align*}
        A_0 &= \{-2m + 2, -2m + 4, \dots, 2m -2\} \\
        A_1 &= \{(k-4)m + 3, (k-4)m + 5, \dots, km - 1\} \\
        A_2 &= \{(k-2)m + 1, (k-2)m + 3, \dots, (k+2)m - 3\}.
    \end{align*}
    For convenience we define
    \begin{equation*}
        A_0' := \{0, 2, \dots, 2m -2\} \cup \{(2k-4)m + 2, (2k-4)m + 4, \dots, (2k-2)m -2\}
    \end{equation*}
    and $A' := A_0' \cup A_1 \cup A_2$, so that $a \in A' + n_{k,m}\mathbb{Z} \iff a \in A + n_{k,m} \mathbb Z$ and $A' \subseteq [0, n_{k, m} - 1]$.

    Next, we consider the distances. Since $i<j$, we have $d_N(i, j) \in \{j - i, N - (j - i)\}$. Noting that $N \equiv 1 \pmod{n_{k,m}}$, we can see that $\varphi_J(v_i) = \varphi_J(v_j)$ implies
    \begin{equation}\notag
        d_N(i, j) \in (A' \cup (n_{k,m} + 1 - A')) + n_{k,m}\mathbb{Z}.
    \end{equation}    
    We define $B:= n_{k, m} + 1 - A'$. As before, we decompose $B = B_0 \cup B_1 \cup B_2$, where $B_0 = n_{k,m} + 1 - A_0'$ and $B_s = n_{k,m} + 1 - A_s$ for $s\in\{1,2\}$. Then
    \begin{align*}
    B_0 &= \{3, \dots, 2m - 1\} \cup \{(2k - 4)m + 3, (2k - 4)m + 5, \dots, 2(k-2)m - 1,2(k-2)m+1\} \\
    B_1 &= \{(k - 2)m + 2, (k - 2)m + 4, \dots, (k + 2)m - 2\}\\
    B_2 &= \{(k - 4)m + 4, (k - 4)m + 6, \dots, km\}.
    \end{align*}

    Noting that $2(k-2)m+1\equiv 1\pmod{n_{k,m}}$, it is further convenient to set $B':=B_0'\cup B_1\cup B_2$, where 
    \begin{equation}\notag
        B_0' := \{1,3, \dots, 2m - 1\} \cup \{(2k - 4)m + 3, (2k - 4)m + 5, \dots, 2(k-2)m - 1\}.
    \end{equation}
Clearly, $(n_{k,m}+1-A')+n_{k,m}\mathbb{Z}=B+n_{k,m}\mathbb{Z}=B'+n_{k,m}\mathbb{Z}$. 
    Now we have
    \begin{equation}\notag
    d_N(i, j) \in (A' \cup B')+ n_{k, m}\mathbb{Z}. 
    \end{equation}
    We note that $A' \cup B'$ is the following union of intervals:
    \begin{equation}\notag
     [0, 2m - 1] \cup [(k - 4)m + 3, (k + 2)m - 2] \cup [(2k - 4)m + 2, (2k-2)m -1]. 
    \end{equation}

    This implies that $d_N(i,j)\notin D_2\cup D_3$. 
    
    It remains to rule out that $d_N(i, j) \in D_1$. Towards a contradiction, suppose that this were the case. Then, since $j-i\in A' + n_{k,m}\Z$ is disjoint from $D_1$, it follows from the above that $d_N(i, j) = N-(j-i)$. 
    
    Let $q_i$ and $q_j$ be the unique integers such that $i = q_in_{k, m} + f_i+t_i$ and $j = q_jn_{k,m}+f_j+t_j $. Then we have $q_j - q_i \leq q - 1$ because $0\leq q_i, q_j < q$ (as $t_i+f_i,t_j+f_j\le (2k-2)m=n_{k,m}$), as well as $t_j - t_i \leq km - 1$ (from (\ref{all-start-diffs})) and $f_j - f_i \leq 2m - 2$ (by definition of $F$). These imply, together with $N = qn_{k,m} + 1$: 
    \begin{equation*}
        d_N(i, j) = N - |i-j| \geq n_{k, m} + 1 - (km - 1) - (2m - 2) = (k-4)m + 4 > 2m,
    \end{equation*}
    where in the last line we used that $k \geq 6$. Hence, we indeed have $d_N(i,j)\notin D_1$, and since we previously showed $d_N(i,j)\notin D_2\cup D_3$, it follows that $d_N(i,j)\notin D$, as desired.

    Summarizing, it follows that there are no adjacent vertices in $G-v_0$ with the the same color under $\phi_J$. This concludes the proof that $\varphi_J$ is a proper $(k - 1)$-coloring of $G - v_0$, also in the case when $k$ is even. 
\end{proof}

\section{Proof of Claim~\ref{structure-claim} for $k \in \{6,8\}$}\label{sec:68}
In Section~\ref{sec:main} we have proved Claim~\ref{structure-claim} for all $k \geq 7$ when $k$ is odd and all $k \geq 10$ when $k$ is even, delaying the proofs for $k=6$ and $k =8$. Here we will provide a proof of Claim~\ref{structure-claim} when $k \in \{6,8\}$. We remark that the proof presented here for these cases could also be extended to all even $k\ge 6$, but we prefer to keep the much shorter proof that works for most values of $k$ presented in Section~\ref{sec:main} and present the more involved argument here only in the cases $k\in \{6,8\}$, where it can be better supported by illustrations.

\begin{claimproof}[Proof of Claim~\ref{structure-claim} when $k \in \{6,8\}$]
We recall that we need to show that there exist an even number $t_{\text{even}} \in \mathbb Z$ and an odd number $t_{\text{odd}} \in \mathbb{Z}$ such that for all $i \in \mathbb{Z}$ we have $\varphi_e(v_i) \neq \varphi_e(v_{i - 2})$ if and only if $i \in \{t_{\text{odd}}, t_{\text{even}}\} + 2m\mathbb{Z}$.

In the following it shall be useful to consider indices of vertices modulo $n_{k,m}$. Thus, we define an equivalence relation $\sim$ on the set $\{v_l|l\in\mathbb{Z}\}$ by $v_{l_1} \sim v_{l_2} \iff l_1 \equiv l_2 \pmod{n_{k,m}}$.

 For every $j \in \mathbb{Z}$ let us denote by $\overline{v_{j}}$ the equivalence class of the element $v_j$ with respect to $\sim$. From Remark~\ref{local-periodicity-rem} we know that all elements of the same equivalence class have the same expected color. Therefore, we can define what we call the \emph{color of an equivalence class} to be the expected color of all its elements, and we denote this color of an equivalence class $\overline{v_i}$ by $\varphi(\overline{v_{i}}) := \varphi_e(v_i)$.

For $i \in \mathbb{Z}$ we say that the integer $i$, vertex $v_i$ or class $\overline{v_i}$ is a \emph{starting point} if $\varphi_e(v_i) \neq \varphi_e(v_{i-2})$. We recall that to prove the claim we must find integers $t_{\text{even}}$ and $t_{\text{odd}}$ such that any $i\in \mathbb{Z}$ is a starting point if and only if $i$ is congruent to $t_{\text{even}}$ or $t_{\text{odd}}$ modulo $2m$.

We now give a brief overview of the rest of the proof: We start by partitioning the vertex classes into sets of size $k-1$ which we call \emph{rings}. Then we will prove statements about the coloring using these rings. In particular, that every class only has two color options, each ring either contains only starting points or only non-starting points, and that every color is the color of exactly one class in each ring. Finally, we will look at the starting points in $2m$ consecutive vertices and prove that there is exactly one even starting point and one odd starting point in that set.

We know that vertices do not have the same color if they are adjacent. We will start by proving a similar statement for the equivalence classes.
        \begin{subclaim}\label{distance-edge-subclaim}
            Let $i,j\in\mathbb{Z}$. If there exists $d \in D$ such that $i - j \in \{d, n_{k,m} - d\} + n_{k,m}\mathbb{Z}$, then $\varphi(\overline{v_i}) \neq \varphi(\overline{v_j})$.
        \end{subclaim}
        \begin{subclaimproof}
            By the definition of $D$, for every $d\in D$ there exists some $d'\in D \cap [1, n_{k,m}]$ with $d\equiv d' \pmod{n_{k,m}}$. Thus, possibly replacing $d$ by $d'$ we may w.l.o.g. assume that $d\in D\cap [1,n_{k,m}]$ in the following. Let $v_{i_0} \in H_{n_{k,m} + 1}$ be such that $i_0 \equiv i \pmod{n_{k,m}}$.
            
            If $i - j \equiv d \pmod{n_{k,m}}$  we define $j_0 := i_0 - d$. Note that $v_{j_0} \in \overline{v_{j}}$, $v_{i_0} \in \overline{v_{i}}$, and $v_{i_0}, v_{j_0} \in S$. As $d_N(i_0, j_0) = d \in D$, we have that $v_{i_0}$ and $v_{j_0}$ are adjacent in $G$ (and thus in $G'$). Therefore, we can conclude $\varphi(\overline{v_i}) = \varphi_e(v_{i_0})=\phi(v_{i_0}) \neq \varphi(v_{j_0})=\varphi_e(v_{j_0})=\varphi(\overline{v_j})$, as desired.
            
            Similarly, if $i - j \equiv n_{k,m} - d \pmod{n_{k,m}}$ we define $j_0' := i_0 + d$. Again we have $v_{j_0'} \in \overline{v_{j}}$ and as before we can conclude from $d_N(i_0, j_0') = d \in D$ and $v_{i_0}, v_{j_0'}\in S$ that $\varphi(\overline{v_i}) \neq \varphi(\overline{v_j})$. 

            This concludes the proof of the claim.
        \end{subclaimproof}

        For every $i \in \mathbb{Z}$ we define the $i$th \emph{ring} as the following set of equivalence classes:
        \begin{equation}\notag
            R_i := \left\{\overline{v_{j}} \mid j\in\mathbb{Z},\,\, j \equiv i \pmod{2m}  \right\}. 
        \end{equation}

        We observe that as $n_{k,m} = 2(k-1)m$, every ring has $k-1$ elements. Furthermore, by definition we have $R_i = R_{i + 2m}$ for all $i \in \mathbb{Z}$, and the sets $R_1, R_2, \dots, R_{2m}$ partition the set of equivalence classes defined by $\sim$. We also define for all $i \in \mathbb{Z}$ a set $Q_i = R_i \cup R_{i+1}$ as the union of two consecutive rings.

        Figure \ref{basic-ring-fig} depicts the set of three consecutive rings $R_0 \cup R_1 \cup R_2$ for $k = 6$ on the left and $k = 8$ on the right. 

        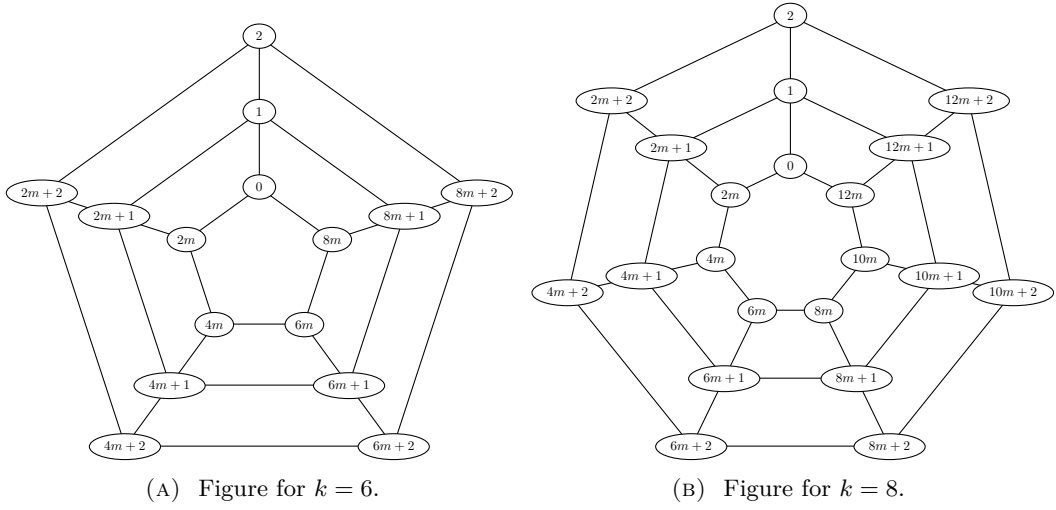
\begin{figure}[ht]

            \begin{subfigure}{0.4\textwidth}
                \begin{tikzpicture}
                    \node[name=l_0, regular polygon, regular polygon sides=5, minimum size = 2cm, draw] at (2,0) {};
                    \node[name=l_1, regular polygon, regular polygon sides=5, minimum size = 4cm, draw] at (2,0) {};
                    \node[name=l_2, regular polygon, regular polygon sides=5, minimum size = 6cm, draw] at (2,0) {};

                    \foreach \x in {1, ..., 5}
                    {
                        \draw (l_0.corner \x) -- (l_1.corner \x);
                        \draw (l_2.corner \x) -- (l_1.corner \x);
                    }

                        \foreach \x in {2,4,...,8}
                        {
                            \pgfmathparse{gcd(\x/2 + 1, 5040)}
                            \node[ellipse, draw, fill=white, scale=0.5] at (l_0.corner \pgfmathresult) {$\x m$};
                            \node[ellipse, draw, fill=white, scale=0.5] at (l_1.corner \pgfmathresult) {$\x m + 1$};
                            \node[ellipse, draw, fill=white, scale=0.5] at (l_2.corner \pgfmathresult) {$\x m + 2$};
                        }              
                        \node[ellipse, draw, fill=white, scale=0.5] at (l_0.corner 1) {\phantom{i}$0$\phantom{i}};
                        \node[ellipse, draw, fill=white, scale=0.5] at (l_1.corner 1) {\phantom{i}$1$\phantom{i}};
                        \node[ellipse, draw, fill=white, scale=0.5] at (l_2.corner 1) {\phantom{i}$2$\phantom{i}};
                    \end{tikzpicture}
                \caption{\label{basic-ring-fig-6} Figure for $k=6$.}
            \end{subfigure}
            \begin{subfigure}{0.4\textwidth}
                \begin{tikzpicture}
                    \node[name=l_0, regular polygon, regular polygon sides=7, minimum size = 2cm, draw] at (2,0) {};
                    \node[name=l_1, regular polygon, regular polygon sides=7, minimum size = 4cm, draw] at (2,0) {};
                    \node[name=l_2, regular polygon, regular polygon sides=7, minimum size = 6cm, draw] at (2,0) {};

                    \foreach \x in {1, ..., 7}
                    {
                        \draw (l_0.corner \x) -- (l_1.corner \x);
                        \draw (l_2.corner \x) -- (l_1.corner \x);
                    }

                    \foreach \x in {2,4,...,12}
                        {
                            \pgfmathparse{gcd(\x/2 + 1, 5040)}
                            \node[ellipse, draw, fill=white, scale=0.5] at (l_0.corner \pgfmathresult) {$\x m$};
                            \node[ellipse, draw, fill=white, scale=0.5] at (l_1.corner \pgfmathresult) {$\x m + 1$};
                            \node[ellipse, draw, fill=white, scale=0.5] at (l_2.corner \pgfmathresult) {$\x m + 2$};
                        }              
                        \node[ellipse, draw, fill=white, scale=0.5] at (l_0.corner 1) {\phantom{i}$0$\phantom{i}};
                        \node[ellipse, draw, fill=white, scale=0.5] at (l_1.corner 1) {\phantom{i}$1$\phantom{i}};
                        \node[ellipse, draw, fill=white, scale=0.5] at (l_2.corner 1) {\phantom{i}$2$\phantom{i}};
                \end{tikzpicture}
                \caption{\label{basic-ring-fig-8} Figure for $k=8$.}
            \end{subfigure}

            \caption{\label{basic-ring-fig} Figures illustrating $R_0 \cup R_1 \cup R_2$. A node with label $x$ corresponds to the equivalence class $\overline{v_{x}}$. Some edges are added between some classes to indicate that, due to Subclaim~\ref{distance-edge-subclaim} and $\{1, 2m\} \subseteq D$, they cannot have the same color.}
        \end{figure}

        Before proving the main statements about the coloring, we start by making observations about vertices of the same color in $Q_i$.

        \begin{subclaim}\label{four-two-subclaim}
            Let $i \in \mathbb{Z}$ and $c := \phi(\overline{v_i})$. Then every class $\overline{v_{j}}$ of color $c$ in $Q_i$ other than $\overline{v_{i}}$ satisfies
            \begin{equation}\notag
                \overline{v_{j}} \in \{\overline{v_{i + (k - 2)m}}, \overline{v_{i + (k - 2)m + 1}}, \overline{v_{i + km}}, \overline{v_{i + km + 1}}\}.
            \end{equation}
            In particular, for every $j \in \mathbb{Z}$, every $c\in [1,k-1]$ is the color of exactly two classes in $Q_j$.
        \end{subclaim}
        \begin{subclaimproof}
            For this proof it is helpful to have Figure~\ref{red-edge-fig} in mind.

            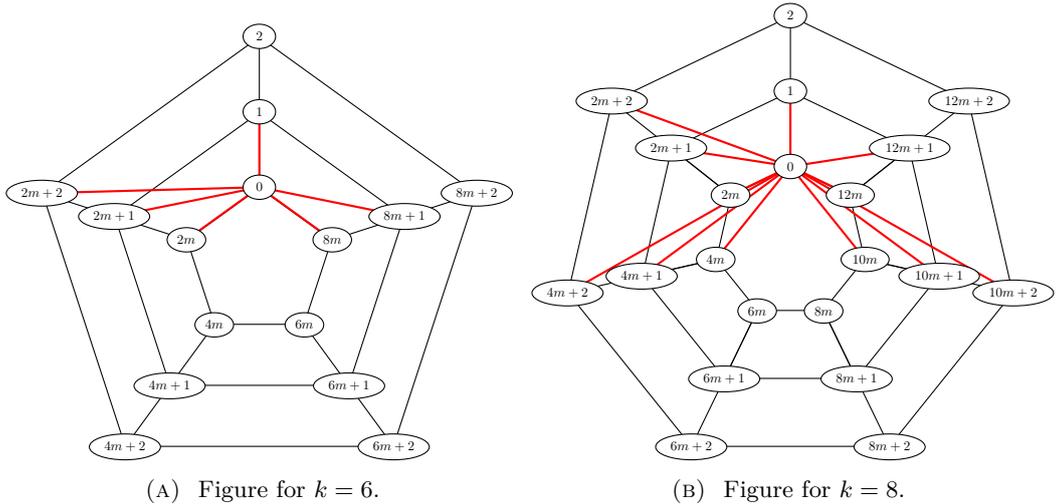
\begin{figure}[H]

                \begin{subfigure}{0.4\textwidth}

                    \begin{tikzpicture}
                        \node[name=l_0, regular polygon, regular polygon sides=5, minimum size = 2cm, draw] at (2,0) {};
                        \node[name=l_1, regular polygon, regular polygon sides=5, minimum size = 4cm, draw] at (2,0) {};
                        \node[name=l_2, regular polygon, regular polygon sides=5, minimum size = 6cm, draw] at (2,0) {};

                        \foreach \x in {1, ..., 5}
                        {
                            \draw (l_0.corner \x) -- (l_1.corner \x);
                            \draw (l_2.corner \x) -- (l_1.corner \x);
                        }

                        \pgfsetlinewidth{2\pgflinewidth}

                        \draw[color=red] (l_0.corner 1) -- (l_1.corner 2);
                        \draw[color=red] (l_0.corner 1) -- (l_1.corner 1);
                        \draw[color=red] (l_0.corner 1) -- (l_1.corner 5);
                        \draw[color=red] (l_0.corner 1) -- (l_2.corner 2);

                        \draw[color=red] (l_0.corner 1) -- (l_0.corner 2);
                        \draw[color=red] (l_0.corner 1) -- (l_0.corner 5);

                        \pgfsetlinewidth{0.5\pgflinewidth}

                        \foreach \x in {2,4,...,8}
                        {
                            \pgfmathparse{gcd(\x/2 + 1, 5040)}
                            \node[ellipse, draw, fill=white, scale=0.5] at (l_0.corner \pgfmathresult) {$\x m$};
                            \node[ellipse, draw, fill=white, scale=0.5] at (l_1.corner \pgfmathresult) {$\x m + 1$};
                            \node[ellipse, draw, fill=white, scale=0.5] at (l_2.corner \pgfmathresult) {$\x m + 2$};
                        }
                        \node[ellipse, draw, fill=white, scale=0.5] at (l_0.corner 1) {\phantom{i}$0$\phantom{i}};
                        \node[ellipse, draw, fill=white, scale=0.5] at (l_1.corner 1) {\phantom{i}$1$\phantom{i}};
                        \node[ellipse, draw, fill=white, scale=0.5] at (l_2.corner 1) {\phantom{i}$2$\phantom{i}};

                    \end{tikzpicture}

                    \caption{\label{red-edge-fig-6} Figure for $k=6$.}
                \end{subfigure}
                \begin{subfigure}{0.4\textwidth}

                    \begin{tikzpicture}
                        \node[name=l_0, regular polygon, regular polygon sides=7, minimum size = 2cm, draw] at (2,0) {};
                        \node[name=l_1, regular polygon, regular polygon sides=7, minimum size = 4cm, draw] at (2,0) {};
                        \node[name=l_2, regular polygon, regular polygon sides=7, minimum size = 6cm, draw] at (2,0) {};

                        \foreach \x in {1, ..., 7}
                        {
                            \draw (l_0.corner \x) -- (l_1.corner \x);
                            \draw (l_0.corner \x) -- (l_2.corner \x);
                        }

                        \pgfsetlinewidth{2\pgflinewidth}

                        \draw[color=red] (l_0.corner 1) -- (l_1.corner 1);
                        \draw[color=red] (l_0.corner 1) -- (l_1.corner 2);
                        \draw[color=red] (l_0.corner 1) -- (l_1.corner 3);
                        \draw[color=red] (l_0.corner 1) -- (l_1.corner 6);
                        \draw[color=red] (l_0.corner 1) -- (l_1.corner 7);
                        
                        \draw[color=red] (l_0.corner 1) -- (l_0.corner 2);
                        \draw[color=red] (l_0.corner 1) -- (l_0.corner 7);
                        \draw[color=red] (l_0.corner 1) -- (l_0.corner 3);
                        \draw[color=red] (l_0.corner 1) -- (l_0.corner 6);

                        \draw[color=red] (l_0.corner 1) -- (l_2.corner 2);
                        \draw[color=red] (l_0.corner 1) -- (l_2.corner 3);
                        \draw[color=red] (l_0.corner 1) -- (l_2.corner 6);

                        \pgfkeys{/pgf/number format/.cd,std,fixed zerofill,precision=0}

                        \pgfsetlinewidth{0.5\pgflinewidth}

                        \foreach \x in {2,4,...,12}
                        {
                            \pgfmathparse{gcd(\x/2 + 1, 5040)}
                            \node[ellipse, draw, fill=white, scale=0.5] at (l_0.corner \pgfmathresult) {$\x m$};
                            \node[ellipse, draw, fill=white, scale=0.5] at (l_1.corner \pgfmathresult) {$\x m + 1$};
                            \node[ellipse, draw, fill=white, scale=0.5] at (l_2.corner \pgfmathresult) {$\x m + 2$};
                        }              
                        \node[ellipse, draw, fill=white, scale=0.5] at (l_0.corner 1) {\phantom{i}$0$\phantom{i}};
                        \node[ellipse, draw, fill=white, scale=0.5] at (l_1.corner 1) {\phantom{i}$1$\phantom{i}};
                        \node[ellipse, draw, fill=white, scale=0.5] at (l_2.corner 1) {\phantom{i}$2$\phantom{i}};
                    \end{tikzpicture}

                    \caption{\label{red-edge-fig-8} Figure for $k=8$.}
                \end{subfigure}

                \caption{\label{red-edge-fig} Figures illustrating with red edges which classes in $R_0\cup R_1 \cup R_2$ must have a color different from that of $\overline{v_0}$ due to Subclaim~\ref{distance-edge-subclaim}. }
            \end{figure}


            We start by noting that $Q_i = \{\overline{v_{i}}, \overline{v_{i+1}}, \overline{v_{i+2m}}, \overline{v_{i+2m+1}}, \dots, \overline{v_{i + (2k - 4)m}}, \overline{v_{i+ (2k - 4)m+1}}\}$. As $1 \in D_1$, $[2m, (k - 4)m + 1] \subseteq D_2$, and $[(k + 2)m, \dots (2k - 4)m + 1] \subseteq D_3$ we can conclude from Subclaim~\ref{distance-edge-subclaim} that if $\overline{v_j} \in Q_i$, $\overline{v_{j}} \neq \overline{v_{i}}$ and $\varphi(\overline{v_j}) = c$ then $\overline{v_{j}} \in \{\overline{v_{i + (k - 2)m}}, \overline{v_{i + (k - 2)m + 1}}, \overline{v_{i + km}}, \overline{v_{i + km + 1}}\}$. This establishes the first part of the claim.
            Next we will prove that for all $j \in \mathbb{Z}$ the set $Q_j$ has exactly $2$ elements of each color. Note that there are $k - 1$ colors in total and $Q_j$ has $2(k - 1)$ elements. Therefore, it suffices to show that $Q_j$ contains at most two classes of each color.

            \noindent So let $c\in [1,k-1]$ be a color. We consider two cases depending on whether $R_j$ contains a class of color $c$.
            
            Suppose first that there is no class of color $c$ in $R_j$, and let $\overline{v_x} \in R_{j + 1}$ be some class of color $c$ (if there is no such class, we are already done). From the already established first part of the claim, we know that the remaining classes of color $c$ in $Q_{j + 1}$ are in 
            \begin{equation}\notag
                \{\overline{v_{x + (k - 2)m}}, \overline{v_{x + (k - 2)m + 1}}, \overline{v_{x + km}}, \overline{v_{x + km + 1}}\},
            \end{equation} 
            of which $\overline{v_{x + (k - 2)m + 1}}$ and $\overline{v_{x + km + 1}}$ are in $R_{j + 2}$. Therefore, all classes except $\overline{v_x}$ in $Q_j$ of color $c$ must be in $\{\overline{v_{x + (k - 2)m}}, \overline{v_{x + km}}\}$. Noting that $km - (k - 2)m = 2m \in D$, we can conclude from Subclaim~\ref{distance-edge-subclaim} that $\varphi(\overline{v_{x + km}}) \neq \varphi(\overline{v_{x + (k - 2)m}})$, so at most one of $\overline{v_{x + (k - 2)m}}, \overline{v_{x + km}}$ can have color $c$. This proves that there are at most two classes of color $c$ in $Q_j$ in this first case, as desired.

            Next let us consider the second case, in which there is at least one class of color $c$ in $R_j$ . Let $\overline{v_x}$ denote such a class. Towards a contradiction, suppose there are more than two classes of color $c$ in $Q_j$, and let $\overline{v_y}\neq \overline{v_z}$ denote two classes of color $c$ in $Q_j\setminus \{\overline{v_x}\}$. Without loss of generality, let $y, z \in x + [1, n_{k,m}]$ and $y < z$. Then we know from the already established first part of the claim that
            \begin{equation}\notag
                y, z \in \{x + (k - 2)m, x + (k - 2)m + 1, x + km, x + km + 1\}.
            \end{equation} 
            Then we have
            \begin{equation}\notag
                z - y \in \{1, 2m - 1, 2m, 2m + 1\}.
            \end{equation}
            In particular, $z - y \in D$ and by Subclaim~\ref{distance-edge-subclaim} we find that $\varphi(\overline{v_y}) \neq \varphi(\overline{v_z})$, contradicting the definitions of $y$ and $z$. This shows that our above assumption was false, indeed $Q_j$ contains at most two classes of color $c$ also in this second case.

            In conclusion, every color appears at most (and as argued above, therefore exactly) two times in every $Q_j$. This concludes the proof of the claim.
        \end{subclaimproof}

        Next we will use Sublcaim~\ref{four-two-subclaim} to restrict the color of every class to just two options.

        \begin{subclaim}\label{two-colors-subclaim}
            Let $i \in \mathbb{Z}$. Then
            \begin{equation}\notag
                \varphi(\overline{v_i}) \in \left\{\varphi(\overline{v_{i - 2}}), \varphi(\overline{v_{i + 2m - 2}})\right\}.
            \end{equation}
        \end{subclaim}
        \begin{subclaimproof}
            Towards a contradiction, suppose that $\varphi(\overline{v_i}) \notin \left\{\varphi(\overline{v_{i - 2}}), \varphi(\overline{v_{i + 2m - 2}})\right\}$ for some $i \in \mathbb{Z}$. We start by showing that there must be a class of color $c := \varphi(\overline{v_i})$ in $R_{i - 2}$: From Subclaim~\ref{four-two-subclaim} we know that both $Q_{i - 2}$ and $Q_{i - 1}$ contain exactly two classes of color $c$. Therefore, $R_{i-2}$ has the same number of classes of color $c$ as $R_i$, which is at least one by definition of $c$. 
            
            Let $\overline{v_{x}}$ be a class of color $c$ in $R_{i - 2}$ and let $\overline{v_{y}}$ be the (unique) other class of color $c$ in $Q_{i - 2}$. We note that 
            \begin{equation}\notag
                [2m + 2, (k - 4)m + 2] \cup [(k + 2)m + 2, (2k - 6)m + 2] \subseteq D. 
            \end{equation} 
            Therefore, applying Subclaim~\ref{distance-edge-subclaim} to the classes $\overline{v}_i, \overline{v}_x$ of the same color and using $i\equiv x + 2 \pmod{2m}$, we find
            \begin{equation}\notag
                i \in x + 2 + \{0, (k - 2)m, km, (2k-4)m\} + n_{k,m}\mathbb{Z}.
            \end{equation}
            By our assumption we have $c=\varphi(\overline{v_i}) \notin \left\{\varphi(\overline{v_{i - 2}}), \varphi(\overline{v_{i + 2m - 2}})\right\}$, and thus $x \notin \{i - 2, i + 2m - 2\} + n_{k,m}\mathbb{Z}$. Hence, we can further refine to
            \begin{equation}\label{i-from-x}
                i \in x + 2 + \{(k - 2)m, km\} + n_{k,m}\mathbb{Z}.
            \end{equation}
            Subclaim~\ref{four-two-subclaim} tells us that 
            \begin{equation}\label{y-from-x}
                y \in x + \{(k - 2)m, (k - 2)m + 1, km, km + 1 \} + n_{k,m}\mathbb{Z}.
            \end{equation}
            Taking the difference between (\ref{i-from-x}) and (\ref{y-from-x}) we obtain
            \begin{align}
                i - y &\in 2 + \{-2m - 1, -2m, -1, 0, 2m - 1, 2m\} + n_{k,m}\mathbb{Z} \\
                &= \{-2m + 1, -2m + 2, 1, 2, 2m + 1, 2m + 2\} + n_{k,m}\mathbb{Z}.\label{i-y}
            \end{align}
            We will now exclude all these options to arrive at a contradiction. We have $i - y \notin \{-2m + 2, 2\} + n_{k,m}\mathbb{Z}$ because our assumption $c=\varphi(\overline{v_i}) \notin \left\{\varphi(\overline{v_{i - 2}}), \varphi(\overline{v_{i + 2m - 2}})\right\}$ implies that $y \notin \{i - 2, i + 2m - 2\} + n_{k,m}\mathbb{Z}$. Since both $\overline{v_y}$ and $\overline{v_i}$ have color $c$, by Subclaim~\ref{distance-edge-subclaim} $i - y$ cannot be congruent to a distance in $D$ modulo $n_{k,m}$. As $\{1, 2m + 1, 2m + 2\} \subseteq D$ and $-2m + 1 \equiv (k - 4)m + 1 \in D$ we can conclude that
            \begin{equation}\notag
                i -y \notin \{-2m + 1, -2m + 2, 1, 2, 2m + 1, 2m + 2\} + n_{k,m}\mathbb{Z},
            \end{equation}
            contradicting (\ref{i-y}). In conclusion, our above assumption was wrong: Indeed, we have $\varphi(\overline{v_i}) \in \{\varphi(\overline{v_{i - 2}}), \varphi(\overline{v_{i + 2m - 2}})\}$ for all $i \in\mathbb{Z}$. This concludes the proof of the claim. 
        \end{subclaimproof}

        The next step can be phrased as saying that for every ring $R_i$, either all of its classes are starting points or none of them are.
        \begin{subclaim}\label{same-choice-subclaim}
            Let $i,j \in \mathbb{Z}$ be such that $i \equiv j \pmod{2m}$. Then $\varphi(\overline{v_i}) = \varphi(\overline{v_{i - 2}})$ if and only if $\varphi(\overline{v_j}) = \varphi(\overline{v_{j - 2}})$.
        \end{subclaim}
        \begin{subclaimproof}
            Towards a contradiction, let $i, j \in \mathbb{Z}$ such that $i \equiv j \pmod{2m}$, $\varphi(\overline{v_i}) \neq \varphi(\overline{v_{i - 2}})$ and $\varphi(\overline{v_j}) = \varphi(\overline{v_{j - 2}})$. By choosing appropriate representatives of the classes, using that $2m$ divides $n_{k,m}$, we may w.l.o.g. assume there exist such $i, j$ satisfying $i < j$. By considering the sequence $\overline{v_i}, \overline{v_{i+2m}},\ldots,\overline{v_j}$, we find that there has to exist an $x \in \mathbb{Z}$ such that $x \geq i$, $x \equiv i \pmod{2m}$, $\varphi(\overline{v_x}) \neq \varphi(\overline{v_{x-2}})$ and $\varphi(\overline{v_{x + 2m}}) = \varphi(\overline{v_{x+2m - 2}})$. Subclaim~\ref{two-colors-subclaim} tells us that as $\varphi(\overline{v_{x}}) \neq \varphi(\overline{v_{x - 2}})$, we have $\varphi(\overline{v_x}) = \varphi(\overline{v_{x + 2m - 2}})=\varphi(\overline{v_{x + 2m}})$. As $2m \in D$, this contradicts Subclaim~\ref{distance-edge-subclaim}. This shows that our above assumption was wrong, and we may conclude that $\varphi(\overline{v_i}) = \varphi(\overline{v_{i - 2}})$ if and only if $\varphi(\overline{v_j}) = \varphi(\overline{v_{j - 2}})$ for all $i, j$ satisfying $i \equiv j \pmod{2m}$, concluding the proof of the claim.
        \end{subclaimproof}

        Next, we will show that every color appears exactly once in each ring.
        \begin{subclaim}\label{each-color-ring-subclaim}
            For every color $c\in [1, k -1]$ and every $x \in \mathbb{Z}$ there exists exactly one class $\overline{v_{i}} \in R_x$ such that $\varphi(\overline{v_{i}}) = c$.
        \end{subclaim}
        \begin{subclaimproof}
            As there are $k-1$ colors and $k-1$ classes in $R_x$, it suffices to show that every color is the color of at most one class of $R_x$. Towards a contradiction, assume there exists a color $c$ and two distinct $\overline{v_i}, \overline{v_j} \in R_x$ such that $\varphi(\overline{v_i}) = c = \varphi(\overline{v_j})$.

            From Subclaim~\ref{four-two-subclaim} we know that 
            \begin{equation}\notag
                j \in i + \{(k-2)m, km\} + n_{k,m}\mathbb{Z}.
            \end{equation}

            If $j \equiv i + km \pmod{n_{k,m}}$ then, as $n_{k,m} = km + (k - 2)m$, we get that $i \equiv j + (k - 2)m \pmod{n_{k,m}}$. So, after swapping $i$ and $j$ if necessary, we may assume without loss of generality that $j \equiv i + (k-2)m \pmod{n_{k,m}}$.
            
            We will show that for all $l \in \mathbb{N}_0$ we have $\varphi(\overline{v_{i + 2l}}) = c = \varphi(\overline{v_{j + 2l}})$ by induction: The base case of $l=0$ is true by definition of $i$ and $j$. Now let $l > 0$ and assume that $\varphi(\overline{v_{i + 2l - 2}}) = c = \varphi(\overline{v_{j + 2l - 2}})$. 

            Consider the class $\overline{v_{j + 2l - 2m}}$. By Subclaim~\ref{two-colors-subclaim} we have 
            \begin{equation}\label{next-color-option-eq}
                \varphi(\overline{v_{j + 2l - 2m}}) \in \{\varphi(\overline{v_{j + 2l - 2m - 2}}), \varphi(\overline{v_{j + 2l - 2}})\}.
            \end{equation}
            By our assumption on $i$ and $j$ we have $j + 2l - 2m \equiv i + 2l + (k - 4)m \pmod{n_{k,m}}$. Using Subclaim~\ref{distance-edge-subclaim} and that $(k - 4)m + 2 \in D_2$ we obtain: 
            \begin{equation}\notag
                \varphi(\overline{v_{j + 2l - 2m}})=\varphi(\overline{v_{i + 2l + (k - 4)m}}) \neq \varphi(\overline{v_{i + 2l - 2}}) = c = \varphi(\overline{v_{j + 2l - 2}}),
            \end{equation}
                where the last two equalities use the induction hypothesis.
                Then we get from (\ref{next-color-option-eq}) that $\varphi(\overline{v_{j + 2l - 2m}}) = \varphi(\overline{v_{j + 2l - 2m - 2}})$. From Subclaim~\ref{same-choice-subclaim} it follows that 
            \begin{equation}\notag
                \varphi(\overline{v_{j + 2l}}) = \varphi(\overline{v_{j + 2l - 2}}) = c 
            \end{equation}
            and
            \begin{equation}\notag
                \varphi(\overline{v_{i + 2l}}) = \varphi(\overline{v_{i + 2l - 2}}) = c.
            \end{equation}

            This establishes the inductive claim for $l$. By the principle of induction, we have now proved that for all $l \in \mathbb{N}_0$ we have $\varphi(\overline{v_{i + 2l}}) = c = \varphi(\overline{v_{j + 2l}})$, so in particular $\varphi(\overline{v_{i + 2m}}) = \varphi(\overline{v_{i}})$. This contradicts Subclaim~\ref{distance-edge-subclaim} because $2m \in D$. Therefore, our initial assumption was wrong, there do not exist two classes in $R_x$ of the same color. As explained at the start, this implies that each color appears exactly once in $R_x$ for every $x \in \mathbb{Z}$, concluding the proof of the claim.
        \end{subclaimproof}

        Finally, we will conclude the proof of Claim~\ref{structure-claim} by showing that there is an even integer $t_{\text{even}} \in \mathbb{Z}$ and an odd integer $t_{\text{odd}} \in \mathbb{Z}$ such that $\varphi_e(v_i) \neq \varphi_e(v_{i - 2})$ if and only if $i \in \{t_{\text{even}}, t_{\text{odd}}\} + 2m\mathbb{Z}$ for every $i\in \mathbb{Z}$. 

        By Subclaim~\ref{same-choice-subclaim} and since $\phi_e(v_i)=\phi(\overline{v_i})$ for every $i\in \mathbb{Z}$, it follows that for every $x,y \in \mathbb{Z}$ such that $x \equiv y \pmod{2m}$ we have $\varphi_e(v_{x}) \neq \varphi_e(v_{x -2})$ if and only if $\varphi_e(v_{y}) \neq \varphi_e(v_{y-2})$. Therefore, it is sufficient to prove that there exists an even and odd number $t_{\text{even}},t_{\text{odd}}\in [2,2m+1]$ such that for all $x \in [2, 2m+1]$ we have $\varphi_e(v_x) \neq \varphi_e(v_{x - 2})$ if and only if $x \in \{t_{\text{even}}, t_{\text{odd}}\}$.

        We can restate this again by splitting the integers by parity: We have to prove that for $x \in \{0, 1\}$ there exists exactly one $t \in x + \{2, 4, \dots, 2m\} = x + 2 + F$ such that $\varphi_e(v_t) \neq \varphi_e(v_{t-2})$.

        Let $x \in \{0, 1\}$ and let $t_1 < t_2 < \dots < t_s$ be all integers $t_i$ in $x + F+2$ satisfying $\varphi_e(v_{t_i}) \neq \varphi_e(v_{t_i-2})$. As stated above, our goal is to prove that $s = 1$. 

        Note that $s \geq 1$: As $2m \in D$, it follows from Subclaim~\ref{distance-edge-subclaim} that $\varphi_e(v_x) \neq \varphi_e(v_{x + 2m})$. Therefore, there must be some $t \in x + 2 + F$ satisfying $\varphi_e(v_t) \neq \varphi_e(v_{t - 2})$.

        It remains to prove that $s \leq 1$. 
        To do so, it will be useful to first understand the (expected) colors of the $v_{t_i}$. 

        For every $i \in [1, s]$, by definition of the $t_i$, we have $\varphi_e(v_{t_i}) \neq \varphi_e(v_{t_i - 2})$. From Subclaim~\ref{two-colors-subclaim} (using $\phi_e(v_x)=\phi(\overline{v_x})$ for all $x\in \mathbb{Z}$) it then follows that $\varphi_e(v_{t_i}) = \varphi_e(v_{t_i + 2m - 2}).$
        Using the definition of the $t_i$ again, we know that for all $y \in x + 2 + F$ such that $y \notin T := \{t_1, \dots, t_s\}$ the expected color satisfies $\varphi_e(v_{y}) = \varphi_e(v_{y -2})$. 

        From Subclaim~\ref{same-choice-subclaim} it follows that for all $y \in x + 2\mathbb{Z}$
        \begin{equation}\label{recursive-colors-eq}
            \varphi_e(v_y) =
            \begin{cases}
                \varphi_e(v_{y - 2}) &\text{if } y \notin T + 2m\mathbb{Z}\\ 
                \varphi_e(v_{y + 2m - 2}) &\text{if } y \in T + 2m\mathbb{Z}.
            \end{cases}
        \end{equation}
        
        For convenience, we set $t_0 := x$ but keep in mind that $t_0 \notin T$.

        From (\ref{recursive-colors-eq}) we obtain for all $i \in [1, s]$, $t \in [t_{i - 1} + 2, t_{i}]$ satisfying $t \in x + 2\mathbb{Z}$ and for all $j \in \mathbb{Z}$: 
        \begin{equation}\notag
            \varphi_e(v_{t + 2mj}) = \varphi_e(v_{t - 2 + 2mj}) = \dots = \varphi_e(v_{t_{i - 1} + 2mj}). 
        \end{equation}

        In particular, setting $t = t_i - 2$ and using~(\ref{recursive-colors-eq}) again, we have, for all $i \in [1, s]$ and all $j \in \mathbb{Z}$:
        \begin{equation}\label{next-line-color-eq}
            \varphi_e(v_{t_i + 2mj}) = \varphi_e(v_{t_i + 2m(j + 1) - 2}) = \varphi_e(v_{t_{i - 1} + 2m(j + 1)}).
        \end{equation}

        By induction on $i$ we obtain from (\ref{next-line-color-eq}) that for all $i \in [0, s]$:
        \begin{equation}\label{starting-colors-eq}
            \varphi_e(v_{t_i}) = \varphi_e(v_{t_0 + 2mi}).
        \end{equation}

        By Subclaim~\ref{each-color-ring-subclaim}, the expected colors in $R_{t_0}$, namely $$\phi_e(v_{t_0})=\phi(\overline{v_{t_0}}), \phi_e(v_{t_0+2m})=\phi(\overline{v_{t_0+2m}}),\ldots,\phi_e(v_{t_0+2m(k-2)})=\phi(\overline{v_{t_0+2m(k-2)}})$$ are pairwise distinct and together cover all the colors in $[1,k-1]$. Using this, the $n_{k,m}=2(k-1)m$-periodicity of $\phi_e$ and~(\ref{starting-colors-eq}), we obtain for all $i\in [0,s]$ and $j\in \mathbb{Z}$:
        \begin{equation}\label{color-mod-k-1-eq}
            \varphi_e(v_{t_i}) = \varphi_e(v_{t_0 + 2mj}) \iff i \equiv j \pmod{k-1}.
        \end{equation}

        Recall that $s \geq 1$. From (\ref{recursive-colors-eq}) it follows that $\varphi_e(v_{t_0 + 2m}) = \varphi_e(v_{t_s})$ because $t_s \leq x+2m=t_0+2m$ by definition and for all $t \in [t_s+2, t_0 + 2m]$ of the same parity as $t_0=x$ (and thus the same parity as $t_s$) we have $\varphi_e(v_t) = \varphi_e(v_{t-2})$. Therefore, we can use (\ref{color-mod-k-1-eq}) to conclude that $s \equiv 1 \pmod{k - 1}$. 
        
        We can now prove that $s \leq 1$: Towards a contradiction, suppose that $s > 1$. Note that since $s\equiv 1 \pmod{k-1}$, this implies $s\ge k$. 
        From Subclaim~\ref{each-color-ring-subclaim} we know that there is exactly one $\overline{v_{y_0}} \in R_{t_0}$ satisfying $\varphi_e(v_{y_0}) = \varphi_e(v_{t_0 + 1})$. Let $j \in [1, k-1]$ be such that $y := t_0 + 2mj \equiv y_0 \pmod{n_{k,m}}$. Then by (\ref{starting-colors-eq}) (which is applicable since $j\le k-1\le s$), we have:
        \begin{equation}\notag
            \varphi_e(v_{t_j}) = \varphi_e(v_{t_0 + 2mj}) = \varphi_e(v_{y_0})=\varphi_e(v_{t_0+1}).
        \end{equation}

        As $t_j \in t_0 + 2 + F$, we have $t_j - (t_0 + 1)\in F+1 =D_1$, contradicting Subclaim~\ref{distance-edge-subclaim}. 

As argued above, this implies the existence of an even integer $t_{\text{even}} \in \mathbb{Z}$ and an odd integer $t_{\text{odd}} \in \mathbb{Z}$ such that $\varphi_e(v_i) \neq \varphi_e(v_{i - 2})$ if and only if $i \in \{t_{\text{even}}, t_{\text{odd}}\} + 2m\mathbb{Z}$ for every $i\in \mathbb{Z}$. This establishes the statement of Claim~\ref{structure-claim}, concluding its proof also in the remaining cases of $k\in \{6,8\}$.

    \end{claimproof}

\end{document}